\documentclass[11pt,a4paper, parskip=half]{scrartcl} 
\usepackage[utf8]{inputenc}
\usepackage[english]{babel}
\usepackage[T1]{fontenc}
\usepackage{lmodern}
\usepackage[left=3cm,right=3cm,top=3cm,bottom=3cm]{geometry}
\usepackage[runin]{abstract}
    \abslabeldelim{.}

\usepackage{todonotes}

\usepackage{amsmath, amsfonts, amssymb, amsthm, amstext}
\usepackage{mathtools}
\usepackage{mathrsfs}
\usepackage{mathdots}
\usepackage{cases}

\usepackage{graphicx}
\usepackage{subcaption}

\font\mfett=cmmib10 at11pt
 at9pt
\def\gamra{\hbox{\mfett\char013}}

\def\lamdra{\hbox{\mfett\char021}}
\def\balpha{\hbox{\mfett\char011}}
\def\blambda{\hbox{\mfett\char021}}

\newcommand\mycom[2]{\genfrac{}{}{0pt}{}{#1}{#2}}

\usepackage{enumitem}

\usepackage{hyperref} 
\usepackage[capitalise]{cleveref}
	
\newcounter{thm}
\numberwithin{thm}{section}
\numberwithin{equation}{section}

	\newtheoremstyle{myplain}		
			{}			
			{}			
			{\itshape}				
			{}				
			{\sffamily\bfseries}				
			{.}		
			{ }				
			{\thmname{#1}\thmnumber{ #2}\textnormal{\textsf{\thmnote{ (#3)}}}}			
    \newtheoremstyle{mybreak}
            {}{}{}{}{\sffamily\bfseries}{.}{\newline}
            {\thmname{#1}\thmnumber{ #2}\textnormal{\textsf{\thmnote{ (#3)}}}}
	\newtheoremstyle{mydef}
			{}{}{}{}{\sffamily\bfseries}{.}{ }
			{\thmname{#1}\thmnumber{ #2}}
	\newtheoremstyle{myrem}
			{}{}{}{}{\sffamily\itshape}{.}{ }
			{\thmname{#1}\thmnumber{ #2}}

\theoremstyle{myplain}
	\newtheorem{theorem}[thm]{Theorem}
	\newtheorem{lemma}[thm]{Lemma}
	
    \newtheorem{corollary}[thm]{Corollary}
\theoremstyle{mybreak}
	\newtheorem{algorithm}[thm]{Algorithm}
\theoremstyle{mydef}
	
	\newtheorem{remark}[thm]{Remark}
\theoremstyle{mydef}
	\newtheorem{example}[thm]{Example}

	\newcommand{\cc}{\mathbb{C}}
		\newcommand{\zz}{\mathbb{Z}}
		\newcommand{\nn}{\mathbb{N}}
	\newcommand{\rr}{\mathbb{R}}

\newcommand{\argmax}{\mathop{\mathrm{argmax}}}

\allowdisplaybreaks

\def\sumprime_#1^#2{
    \setbox0=\hbox{$\scriptstyle{#1}$}
    \setbox1=\hbox{$\scriptstyle{#2}$}
    \setbox2=\hbox{$\displaystyle{\sum}$}
    \setbox4=\hbox{${}^\prime\mathsurround=0pt$}
    \dimen0=.5\wd0 \advance\dimen0 by-.5\wd2
    \ifdim\dimen0>0pt
        \ifdim\dimen0>\wd4 \kern\wd4
        \else\kern\dimen0
        \ifdim\dimen1>\wd4 \kern\wd4
        \else\kern\dimen1
    \fi\fi\fi
\mathop{{\sum}^\prime}_{\kern-\wd4 #1}^{\kern-\wd4 #2}
}

\title{\Large Exact Reconstruction of Extended Exponential Sums Using Rational Approximation of their Fourier Coefficients}
\author{Nadiia Derevianko\footnote{Institute for Numerical and Applied Mathematics, G\"ottingen University, Lotzestr.\ 16-18, 37083 G\"ottingen, Germany, \{n.derevianko,plonka\}@math.uni-goettingen.de} \footnote{Corresponding author} \quad Gerlind Plonka$^{*}$}
\date{}

\begin{document}
	\let\oldproofname=\proofname
	\renewcommand{\proofname}{\itshape\sffamily{\oldproofname}}

\maketitle

\begin{abstract}
In this paper, we derive a new recovery procedure for the reconstruction of extended exponential sums of the form 
$y(t) = \sum_{j=1}^{M} \left( \sum_{m=0}^{n_j}  \, \gamma_{j,m} \, t^{m} \right) {\mathrm e}^{2\pi \lambda_j t}$,  where the frequency parameters $\lambda_{j} \in {\mathbb C}$ are pairwise distinct.  For the reconstruction  we employ a finite set of  classical Fourier coefficients of $y$ with regard to a finite interval $[0,P] \subset {\mathbb R}$ with $P>0$. Our method requires 
at most $2N+2$ Fourier coefficients $c_{k}(y)$ to recover all parameters of $y$, where $N:=\sum_{j=1}^{M} (1+n_{j})$ denotes the order of $y(t)$.
The recovery is based on the observation that for $\lambda_{j} \not\in \frac{{\mathrm i}}{P} {\mathbb Z}$ the terms of $y(t)$ possess Fourier coefficients with rational structure. We employ a recently proposed stable iterative rational approximation algorithm in \cite{AAA}. 
If a sufficiently large set of $L$ Fourier coefficients of $y$ is available (i.e., $L > 2N+2$), then our recovery method automatically detects   
 the number $M$ of terms of $y$, the multiplicities $n_{j}$ for $j=1, \ldots , M$, as well as all parameters $\lambda_{j}$, $j=1, \ldots , M$ and $ \gamma_{j,m}$ $j=1, \ldots , M$, $m=0, \ldots , n_{j}$, determining $y(t)$. Therefore our method provides a new stable alternative to the known numerical approaches for the recovery of exponential sums that are based on Prony's method.\\
\textbf{Keywords:}  sparse exponential sums, extended exponential sums, rational approximation, AAA algorithm, barycentric representation,
Fourier coefficients.\\
\textbf{AMS classification:}
41A20, 42A16, 42C15, 65D15, 94A12.
\end{abstract}

\section{Introduction}
\label{introduction}

Recently, we have proposed a new reconstruction method to recover real functions of the form 
$$
y(t) =\sum \limits_{j=1}^{N} \gamma_j \cos(2\pi a_j t +b_j), \qquad  \gamma_j\in(0,\infty), \ (a_j,b_j)\in (0,\infty)\times [0,2\pi),
$$
from a limited number of classical Fourier coefficients of $y$ from its Fourier expansion on a given fixed interval $[0, P]$, see \cite{PP2020}.

This paper continues and strongly generalizes our research started in \cite{PP2020}. We present a new reconstruction method to recover 
complex extended exponential sums, i.e., sums being of polynomial exponential form,
$$ y(t)=\sum_{j=1}^M \left( \sum_{m=0}^{n_j}  \, \gamma_{j,m} \, t^{m} \right) {\mathrm e}^{\lambda_jt}, \quad \gamma_{j,m} \in \cc, \, \gamma_{j,n_{j}} \neq 0, \,  \lambda_j \in \cc.  
$$

First we introduce the main objects studied in the paper. 
For $N \in \nn$, we consider the set of \textit{proper exponential sums}
\begin{equation}\label{yn0}
\mathcal{Y}_N^{0}=\left\{  y: \, y(t)= \sum\limits_{j=1}^{n} \gamma_j {\mathrm e}^{\lambda_jt}, \ \  \gamma_j, \lambda_j \in \cc ,  \, n\leq N \right\}
\end{equation}
and the set of \textit{extended exponential sums }
\begin{equation}\label{yn}
\mathcal{Y}_N=\left\{  y: \, y(t)=\sum_{j=1}^M \left( \sum_{m=0}^{n_j}  \, \gamma_{j,m} \, t^{m} \right) {\mathrm e}^{\lambda_jt}, \ \ \gamma_{j,m}, \lambda_j \in \cc,  \, n:=\sum\limits_{j=1}^{M}(1+n_j ) \leq N\right\}.
\end{equation}
We call $n=n(y)$ the \textit{order} of the exponential sum $y(t)$ and $M=M(y)$ its \textit{length}. Obviously, $\mathcal{Y}_N^{0}$ is a subset of $\mathcal{Y}_N$ and we have $n(y)\geq M(y)$ in the general case and  $n(y)=M(y)$ for $y \in \mathcal{Y}_N^{0}$.
In fact, $\mathcal{Y}_N^{0}$ is a dense subset of $\mathcal{Y}_N$, and $\mathcal{Y}_N$ is closed with respect to the maximum norm in $C[a,b]$ for any compact interval $[a,b]$. Moreover,  $\mathcal{Y}_N$ is an existence set for the space $C[a,b]$ of continuous functions given on a compact interval $[a,b]$,  see,  \cite{Rice62} or \cite[Chapter VI]{Bbook}. 
Approximation with extended exponential sums has a long history. We refer to 
\cite{Kam76}, \cite{Wel81} that are dedicated to the question of approximation of classes of smooth functions $L_{p}[a,b]$, $1\leq p < \infty$, and $C[a,b]$ by extended exponential sums. Further, it is well-known for a long time that there is a close connection between approximation with exponential sums and rational approximation, see \cite{Bbook, Sidi82}.

The extended exponential sums $y(t)$ in $\mathcal{Y}_N$ of order $n \le N$ are solutions of homogeneous linear differential equations of order $n$ with constant coefficients $a_{0}, \ldots , a_{n-1} \in {\mathbb C}$ of the form
\begin{equation}\label{eq}
 y^{(n)}+a_{n-1} y^{(n-1)}+ \ldots +a_{1} y'+a_0 y=0.
\end{equation}
The corresponding characteristic polynomial is given as
$$
p(\lambda)= \lambda^{n}+a_{n-1} \lambda^{n-1}+ \ldots +a_{1} \lambda+a_0 = \prod_{j=1}^{M} (\lambda - \lambda_{j})^{n_{j}+1},
$$
i.e., $p(\lambda)$ has $M$ distinct roots  $\lambda_j$ with multiplicity $n_j+1$, $j=1, \ldots ,M$. 
In particular, the functions $t^{m} {\mathrm e}^{\lambda_jt}$, $m=0,\ldots ,n_j$, $j=1,\ldots ,M$, are linearly independent and form a basis of the space of the solutions of the differential equation (\ref{eq}). Similarly, it can be shown that $\mathcal{Y}_N$ is the solution space for homogeneous linear difference equations  of order $n \le N$, see e.g.\ \cite{Berg}, and therefore $\mathcal{Y}_N$ is closely related to the characterization of Hankel operators of finite rank, \cite{HR84, Peller}.

Extended exponential sums appear in many applications in system identification  and sparse approximation, see e.g. \cite{BDR06,PS10}.
For a comprehensive study of extended exponential sums from an algebraic point of view, we refer to \cite{Mo18}.

Our goal in this paper is to reconstruct the exponential sums in $\mathcal{Y}_N^{0}$ and $\mathcal{Y}_N$. 
This problem has been extensively studied  using Prony's method and its generalizations, see for example \cite{CL2020}, \cite{FMP2012}, \cite{PPT2011},  \cite{PT2010},   \cite{PSK2019},  \cite{PT2014}. However, the known studies mainly focussed  on the recovery of proper exponential sums. The extended model is less often treated, see \cite{Bat13,Bat17, PP2013, PT2013, Sidi82, Sidi85}.
Since Prony’s method involves computations with  Hankel or Toeplitz matrices with possibly high condition numbers, it requires a very careful numerical treatment. This is particularly true if the distance between two distinct frequency parameters $\lambda_{j_{1}}$ and $\lambda_{j_{2}}$ is very small. The extended exponential sums appear if such frequencies collide.

Our new method for reconstruction of signals in form of (extended) exponential sums is based on rational approximation and can be seen as a good alternative to the Prony reconstruction approach.
As input information for our new  algorithm we use a finite set of their Fourier coefficients. 
More precisely,  we consider the Fourier expansion of $y \in \mathcal{Y}_N$ on $[0,P]$ for some $P$ with  $0<P < \infty$ of the form 
$$ y(t) = \sum_{k \in {\mathbb Z}} c_{k}(y) \, {\mathrm e}^{2\pi {\mathrm i} kt/P} $$
with Fourier coefficients $c_{k}(y) := \frac{1}{P} \int_{0}^{P} y(t) \, {\mathrm e}^{-2 \pi {\mathrm i} kt/P} \, {\mathrm d} t$. We will derive algorithms to reconstruct $y$ from a sufficiently large set of given Fourier coefficients $c_{k}(y)$. 
Although $\mathcal{Y}_N^{0}$ is a subset of $\mathcal{Y}_N$ we will separate the reconstruction of  proper exponential sums as a special case due to the fact that in many applications only the set $\mathcal{Y}_N^{0}$ is considered. 

If we have $-{\mathrm i}\lambda_{j} \not\in \frac{1}{P} {\mathbb Z}$ for all frequencies $\lambda_{j}$, i.e., if $y(t)$ in (\ref{yn0}) or (\ref{yn})  does not possess any $P$-periodic terms, then we will employ the special property that the  classical Fourier coefficients $c_k(y)$ of  $y \in \mathcal{Y}_N^{0}$ and  $y \in \mathcal{Y}_N$  of full order $N$  can usually be represented by a rational function $r_{N}$ of type $(N-1,N)$, i.e., $c_k(y)=r_N(k)$ for $k \in {\mathbb Z}$. 
Our reconstruction approach is then based on the recovery of this rational function $r_{N}$  using a  modification of the  AAA algorithm that has been recently proposed in \cite{AAA}. Numerical stability of this algorithm is ensured by barycentric representation of the constructed approximant. The second important property of the AAA algorithm is that it works iteratively thereby  enlarging the polynomial degree of the numerator and denominator of the rational approximant step by step.
The  modified AAA algorithm will terminate if the needed order $N$ of the exponential sum is reached. This gives us the opportunity to reconstruct the order $N$ of the exponential sum, supposed that  a sufficiently large set of $L\ge 2N+1$ Fourier coefficients is given.

If  $y \in \mathcal{Y}_N$ is an extended exponential sum, then the corresponding rational function $r_N$ has $M$ multiple poles with multiplicities $n_j+1$ such that $\sum_{j=1}^{M}(1+n_j)=N$. 
Having reconstructed the rational function $r_{N}(t)$, we will show, how all wanted parameters that determine $y \in \mathcal{Y}_N$ can be uniquely computed from $r_{N}$. 
More exactly, beside $N$, $M$, and $n_{j}$, $j=1, \ldots , M$, we can determine all frequencies $\lambda_{j}$ and all polynomial coefficients $\gamma_{j,m}$ $j=1, \ldots , M$, $m=0, \ldots , n_{j}$, from $r_{N}$.

Further we will extend the algorithm for the case when $P$-periodic components also appear in  the expansions  (\ref{yn0}) and (\ref{yn}). In this case we first reconstruct the rational function that determines the non-$P$-periodic part of $y(t)$ using modified AAA-algorithm and recover all parameters of the non-$P$-periodic part of $y(t)$ as before. In a second step, the  $P$-periodic part of $y(t)$ is determined. 
Note that this is only possible if the given set of Fourier coefficients contains all $c_{k_j}(y)$ that are related to the occurring frequencies that have to be recovered, i.e., we need knowledge about $c_{k_j}(y)$
if  $\lambda_j= \frac{{\mathrm i} k_j}{ P}$ is a frequency parameter that occurs in $y(t)$. 
For proper exponential sums with periodic terms it can be simply shown that only a finite number of Fourier coefficients of $y(t)$ loses its special rational structure.
The recovery of $P$-periodic parts for (truly) extended exponential sums requires a special treatment, since in this case the polynomial coefficients corresponding to $P$-periodic terms still influence all Fourier coefficients of $y$.

\textbf{Outline.} In Section 2 we prove that the signals $y \in \mathcal{Y}_N$ (and consequently $y  \in \mathcal{Y}_N^{0}$) are uniquely determined by given parameters $M$, $n_j$, $\lambda_j$, $\gamma_{j,m}$ for $j, \ldots ,M$, $m=0, \ldots ,n_j$. 
In Section 3 we shortly describe the main ideas of the needed modified AAA-algorithm for rational approximation of functions.  
In Section 4 we consider the recovery of proper exponential sums. We separate two possible cases. In Section 4.1 we study the recovery of  proper exponential sums that contain only non-$P$-periodic terms,  and Section 4.2 is dedicated to the case when  $y \in \mathcal{Y}_N^{0}$ contains also $P$-periodic components.  In Section 5 we consider the recovery of extended exponential sums. In Section 5.1 we show that the Fourier coefficients of $y \in \mathcal{Y}_{N}$ can be represented by a rational function, and we show, how the wanted parameters can be reconstructed from this rational representation.
 Again we separate the consideration of the following two cases: In Section 5.2 we study the  reconstruction of extended exponential sums containing only non-$P$-periodic components.
In Section 5.3, we assume that extended exponential sums contain also $P$-periodic terms. The proposed reconstruction algorithms in Sections 4 and 5  are illustrated by numerical examples. 
The corresponding software can be found on our homepage http://na.math.uni-goettingen.de  We conclude the paper with some final remarks in Section 6.

\textbf{Notation.} As usual $\nn$, $\zz$, $\rr$ and $\cc$ are reserved for natural, integer, real and complex numbers and  let $\nn_0:=\nn \cup \{0\}$.

\section{Uniqueness of the Representation of Extended Exponential Sums}
\label{uniqueness}

In order to be able to reconstruct functions $y(t)$ from $\mathcal{Y}_N$ uniquely, we need to ensure that all parameters of 
\begin{equation}\label{y1} y(t)=\sum\limits_{j=1}^{M} p_{n_j}(t) \, {\mathrm e}^{\lambda_j t} = \sum_{j=1}^{M} \left( \sum_{m=0}^{n_{j}} \gamma_{j,m} \, t^{m} \right) \, {\mathrm e}^{\lambda_j t} 
\end{equation}
can be  uniquely determined.  For this purpose we assume the following restrictions on the parameters:
\begin{description}
\item{(1)} $ M \in \nn$, $M<\infty$;
\item{(2)} $ n_{j} \in  \nn_0$, $n_j<\infty$;
\item{(3)} $\gamma_{j,m} \in {\mathbb C}$ for $j=1, \ldots , M$ and $m=0, \ldots , n_{j}$, and $\gamma_{j,n_{j}} \neq 0$ for $j=1, \ldots , M$;
\item{(4)} $\lambda_{j} \in {\mathbb C}$ are pairwise distinct for $j=1, \ldots , M$.
\end{description}

Conditions (3) and (4) do not restrict the set $\mathcal{Y}_N$ with $N \ge \sum_{j=1}^{M}(1 + n_{j})$. 
If $\gamma_{j,n_{j}}=0$, then  the corresponding term in (\ref{y1})  can be removed, and if $\lambda_{j_{1}} = \lambda_{j_{2}}$ for two frequency parameters in (\ref{y1}), the corresponding terms can be combined into one term.
Note that the third condition implies that the polynomial $p_{n_j}(t)$ has exactly the degree $n_j$. 
In the special case $y \in  \mathcal{Y}_N^{0}$, the restriction (2) simplifies to $n_{j}=0$ for $j=1, \ldots , M$, and restriction (3) reads $\gamma_{j} = \gamma_{j,0} \neq 0$ for $j=1, \ldots , M$. We show now that with the restrictions above, all parameters of $y(t)$ are uniquely determined.

\begin{theorem}\label{uniq}
Let two extended exponential sums be of the form 
$$
 y_1(t)=\sum_{j=1}^{M_1} \left( \sum_{m=0}^{n'_j}  \, \gamma'_{j,m} \, t^{m} \right) {\mathrm e}^{\lambda'_jt} \; \text{and} \quad y_2(t)=\sum_{j=1}^{M_2} \left( \sum_{m=0}^{n''_j}  \, \gamma''_{j,m} \, t^{m} \right) {\mathrm e}^{\lambda''_jt},
$$
and assume that the parameters for $y_{1}(t)$ and $y_{2}(t)$ satisfy the restrictions $(1)- (4)$ given above.
If $y_1(t)=y_2(t)$ pointwise on a finite interval with positive length, then $M_1=M_2 =:M$  and (after suitable permutation of summands)  $n'_j=n''_j =:n_{j}$, $\lambda'_j =\lambda''_j$ for $j=1, \ldots , M$, and   $\gamma'_{j,m} =\gamma''_{j,m} $ for $j=1,\ldots ,M$, $m=0,\ldots ,n_j$.
\end{theorem} 

\begin{proof}
We consider the function 
\begin{equation} \label{rep-y1}
y(t)=y_1(t)-y_2(t)=\sum\limits_{j=1}^{M_1+M_2} \left( \sum_{m=0}^{n_j}  \, \gamma_{j,m} \, t^{m} \right) {\mathrm e}^{\lambda_jt} 
\end{equation}
with
\begin{align*}
\lambda_j & :=\lambda'_j, \quad\quad  n_j :=n'_j, \quad\quad     \gamma_{j,m}  :=\gamma'_{j,m}  \quad   \text{for} \; m=0,\ldots ,n_j \quad \quad \text{and} \;  j=1,\ldots ,M_1 ; \\
\lambda_{M_1+j} & :=\lambda''_j, \,  n_{M_1+j} :=n''_j, \;   \gamma_{M_1+j,m}  :=-\gamma''_{j,m}  \,   \text{for}  \; m=0,\ldots ,n_{M_1+j} \;  \text{and} \;  j=1,\ldots,M_2.
\end{align*}
Let $y_1(t)=y_2(t)$ for $t\in I$, where $I\subset \rr$ is some interval with finite positive length. Then $y\equiv 0$ on $I$.

Let $M$ denote the number of pairwise distinct frequency parameters $\lambda_j$ in the representation (\ref{rep-y1}) of the function $y(t)$.  According to the restriction (4) on the frequency parameters  we have $M\geq \max\{M_1, \, M_2\}$. Then we can rewrite 
\begin{equation} \label{repr-y}
y(t)=\sum\limits_{\ell=1}^{M} \left( \sum_{m=0}^{n_\ell}  \, \tilde{\gamma}_{\ell,m} \, t^{m} \right) {\mathrm e}^{\tilde{\lambda}_\ell t}, 
\end{equation}
where $\tilde{\lambda}_\ell \in \Lambda:=\{  \lambda_j: \, j=1,\ldots,M_1+M_2 \}$ are pairwise distinct. The functions $t^{m} {\mathrm e}^{\tilde{\lambda}_\ell t} $, $m=0,\ldots,n_\ell$, $\ell =1,\ldots,M$,  are linearly independent as the solutions of the equation (\ref{eq}) with $n=\sum_{\ell=1}^{M}(1+n_\ell )$. Therefore, the  condition 
\begin{equation}\label{plus1}
y(t) = \sum_{\ell=1}^{M} \left( \sum_{m=0}^{n_\ell}  \, \tilde{\gamma}_{\ell,m} \, t^{m} \right) {\mathrm e}^{\tilde{\lambda}_\ell t}=0, \qquad  t \in I \subset \rr,
\end{equation}
yields $\tilde{\gamma}_{\ell,m}=0$ for all $\ell=1,\ldots,M$ and $m=0,\ldots,n_\ell$.

Each parameter $\tilde{\lambda}_\ell$ can occur once or twice in the set $\Lambda$. If it occurs once, for example $\tilde{\lambda}_\ell=\lambda_j'$, 
then (\ref{plus1}) implies that the corresponding polynomial coefficient in  $p_{n_\ell}(t) = \sum_{m=0}^{n_{\ell}} \tilde{\gamma}_{\ell,m} t^{m}$ completely vanishes, i.e.,  $\tilde{\gamma}_{\ell,m}=0$ for $m=0, \ldots, n_{\ell}$. But  
this contradicts the assumption (3).
 Therefore, each $\tilde{\lambda}_\ell$ occurs twice in the set $\Lambda$ and we can conclude that  $M_1=M_2 =M$. 
 Further, for each $\ell=1, \ldots , M$,  we find $j_{1}, \, j_{2} \in \{1, \ldots , M\}$  such that 
$\tilde{\lambda}_\ell=\lambda_{j_1}'=\lambda_{j_2}''$. Then the  polynomial coefficient corresponding to ${\mathrm e}^{\tilde{\lambda}_\ell t}$  in (\ref{repr-y}) satisfies
$$
p_{n_\ell}(t) =\sum_{m=0}^{n_{\ell}}  \, \tilde{\gamma}_{\ell,m} \, t^{m} = \sum_{m=0}^{n_{j_1}'}  \, \gamma_{j_1,m}' \, t^{m}-\sum_{m=0}^{n_{j_2}''}  \, \gamma_{j_2,m}'' \, t^{m} =0.
$$
Since by restriction (3), $\gamma_{j_1,n_{j_{1}}'}' \neq 0$ and $\gamma_{j_2,n_{j_{2}}''}''\neq 0$, we conclude that $n_{j_1}'=n_{j_2}''=:n_{j}$ and  $\gamma_{j_1,m}'=\gamma_{j_2,m}''$ for $m=0, \ldots, n_{j}$. 
\end{proof}

\begin{remark}\label{uniq1} 
If  two proper  exponential sums
$$
 y_1(t)=\sum_{j=1}^{M_1}  \gamma'_{j} {\mathrm e}^{\lambda'_jt} \qquad  \text{and} \qquad   y_2(t)=\sum_{j=1}^{M_2}  \gamma''_{j}  {\mathrm e}^{\lambda''_jt},
$$
satisfy the restrictions (1) - (4) and are identical on an interval of finite positive length, then Theorem \ref{uniq} implies that 
 $M_1=M_2=M$  and (after suitable permutation of summands)  $\lambda'_j =\lambda''_j$ and   $\gamma'_{j} =\gamma''_{j} $ for $j=1,\ldots,M$.
\end{remark}


\section{The modified AAA Algorithm for Rational Approximation} 
\label{aaa-algorithm}

Our reconstruction algorithms for exponential sums in Sections \ref{proper-exp} and \ref{extended-exp} are essentially based on the following slight modification of the AAA algorithm in \cite{AAA}.
For the convenience of the reader  we shortly summarize this algorithm with the needed slight modification for rational functions of type 
 $(N-1,N)$. For a similar modification for special point sets we refer to \cite{PP2020}. 
   
For a given sufficiently large finite set of pairwise distinct points $\Gamma \subset {\mathbb R}$ and a corresponding set of values $\{ f(z) : \, z \in \Gamma \}$, the (modified) AAA algorithm iteratively computes a rational function $r_N$ of type $(N-1,N)$ such that $r_N(z) =  f(z)$ for $z \in S$, where $S \subset \Gamma$ is a subset of $N+1$ given points, and such that the error $|r_{N}(z) -  f(z)| $ is small for the remaining points $z \in \Gamma \setminus S$.

At the iteration step $J \ge 1$, we proceed as follows.
Assume that we have  the  set $S_{J+1}=\{z_1, \ldots , z_{J+1}\} \subset \Gamma$, where we want to interpolate $f(z_{j})$, and let $\Gamma_{J+1} := \Gamma \setminus  S_{J+1}$ be the point set where we will approximate. We introduce the corresponding vectors
$$\mathbf f_{S_{J+1}}:= \left( f(z_{j}) \right)_{j=1}^{J+1} \in {\mathbb C}^{J+1}, \qquad  
\mathbf f_{\Gamma_{J+1}} := \left( f(z) \right)_{z \in \Gamma_{J+1}} \in {\mathbb C}^{L-1-J}.$$ 
The rational function $r_{J}(z)$ is constructed in  barycentric form $r_{J}(z):  = \tilde{p}_{J}(z)/\tilde{q}_{J}(z)$ with 
\begin{equation}\label{bar-form}
\tilde{p}_{J}(z) := \sum_{j=1}^{J+1} \frac{w_j \, f(z_{j})}{z-z_j}, \qquad  \tilde{q}_{J}(z) := \sum_{j=1}^{J+1} \frac{w_j}{z-z_j}, 
\end{equation}
where $w_j \in {\mathbb C}$, $j=1, \ldots , J+1$, are weights. This representation already implies that  the interpolation conditions $r_{J}(z_j) = f(z_{j})$ are satisfied for $w_{j} \neq 0$, $j \in \{1, \ldots, J+1\}$.
The vector of weights ${\mathbf w} := (w_1, \ldots , w_{J+1})^T$ is now chosen such that $r(z)$ approximates the remaining data and additionally satisfies 
the side conditions  
\begin{equation}\label{condw}
\| {\mathbf w} \|_2^2 = \sum_{j=1}^{J+1} w_j^2 = 1 \quad \text{and} \quad {\mathbf w}^T  \,  \mathbf f_{S_{J+1}}= \sum_{j=1}^{J+1} w_j f(z_{j}) = 0.
\end{equation}
The first condition is a normalization condition. The second condition in (\ref{condw}) ensures that $r_{J}(z)$ is of a type $(J-1,J)$. 
To compute ${\mathbf w}$, we consider the restricted least-squares problem  
\begin{equation}\label{mini}
\min_{{\mathbf w}} \sum_{z \in \Gamma_{J+1}} \left|f(z) \, \tilde{q}_{J}(z)-\tilde{p}_{J}(z)\right|^2, \quad \textrm{s.t.} \quad  \| {\mathbf w} \|_2^2 = 1,  \;   {\mathbf w}^T  \,  \mathbf f_{S_{J+1}}=0.
\end{equation}
At the first iteration step $J=1$, ${\mathbf w} \in {\mathbb C}^{2}$ is already completely fixed by the two side conditions. For $J>1$, we define the matrices
$$ 
{\mathbf A}_{J+1}:= \left( \frac{f(z) - f(z_{j})}{z-z_j} \right)_{z \in \Gamma_{J+1}, z_j \in S_{J+1}},\quad 
{\mathbf C}_{J+1}:= \left( \frac{1}{z-z_j} \right)_{z \in \Gamma_{J+1}, z_j \in S_{J+1}},$$
and rewrite the term in (\ref{mini}) as
$$
\sum_{z \in \Gamma_{J+1}} \left|f(z) \, \tilde{q}_{J}(z)-\tilde{p}_{J}(z)\right|^2 = \sum_{z \in \Gamma_{J+1}} \left|       {\mathbf w}^T \, \left( 
\frac{f(z)- f(z_{j})}{z-z_j} \right)_{j=1}^{J+1} \right|^2 
= \| {\mathbf A}_{J+1} {\mathbf w} \|_2^2. 
$$
Now,  the minimization problem in (\ref{mini}) takes the form
\begin{equation}\label{mini1}
\min_{{\mathbf w}} \| {\mathbf A}_{J+1} {\mathbf w} \|_2^2 \quad \textrm{s.t.} \quad   \| {\mathbf w} \|_2^2 = 1,  \;  {\mathbf w}^T  \,  \mathbf f_{S_{J+1}}=0.
\end{equation}

To find the solution vector ${\mathbf w} = {\mathbf w}_{J+1}$  of  (\ref{mini1}) approximately, we compute the right (normalized) singular vectors ${\mathbf v}_1$ and  ${\mathbf v}_2$ of the matrix ${\mathbf A}_{J+1}$ corresponding to the two smallest singular values $\sigma_1 \leq \sigma_2$ of ${\mathbf A}_{J+1}$ and take 
$$
{\mathbf w}_{J+1}= \frac{1}{\sqrt{({\mathbf v}_1^{T} \mathbf f_{S_{J}})^{2}+({\mathbf v}_2^{T} \mathbf f_{S_{J}})^{2}}} \left( ({\mathbf v}_2^{T} \mathbf f_{S_{J}}){\mathbf v}_1 - ({\mathbf v}_1^{T} \mathbf f_{S_{J}}){\mathbf v}_2 \right),
$$
such that  $\| {\mathbf w}_{J+1} \|_2^2 = 1$ and ${\mathbf w}_{J+1}^T  \,  \mathbf f_{S_{J+1}}=0$. 
Having determined the weight vector ${\mathbf w}_{J+1}$, the rational function $r_{J}$ is completely fixed by (\ref{bar-form}).
Finally we  consider the errors $|r_{J}(z) - f(z)|$ for all $z \in \Gamma_{J+1}$, where we do not interpolate.
The algorithm terminates if $\max_{z \in \Gamma_{J+1}} | r_{J}(z) - f(z)| < \epsilon$ for a predetermined bound $\epsilon$ or if $J$ reaches a predetermined maximal degree. Otherwise, we find the next point for interpolation as
$$ z_{J+2} := \argmax_{z \in \Gamma_{J+1}} | r_{J}(z) - f(z)|. $$

\begin{algorithm}[Modified AAA algorithm] 
\label{alg1} 
\textbf{Input: } \\
$\mathbf {\Gamma} \in \cc^{L}$ set of given support points $z_j$, $j=1,\ldots,L$, with $L$ large enough \\
$\mathbf{f} \in \cc^{L}$ vector of given function values $f(z_{j})$ corresponding to $\Gamma$\\
\textit{tol>0} tolerance for the approximation error \\ 
\textit{jmax} $\in \nn$ with $\textit{jmax} < \lfloor \frac{L-1}{2} \rfloor$ maximal order of polynomials in the rational function

\noindent
\textbf{Main Loop:}

\noindent
for $j=1:$ \textit{jmax}
\begin{itemize}
\item If $j=1$, choose ${\mathbf S}:=(z_{1}, z_{2})^{T}$, $\mathbf{f}_{\mathbf{S}} := (f(z_{1}), f(z_{2}))^{T}$, where $f(z_{1})$ and $f(z_{2})$ from $\mathbf{f}$ have largest absolute values; update $\mathbf{\Gamma}$ and $\mathbf{f}$ by deleting $z_{1}, z_{2}$ in $\mathbf{\Gamma}$ and $f(z_{1}), f(z_{2})$ in $\mathbf{f}$.\\
 If $ j>1$, compute $z_{k} := \argmax_{z \in \mathbf{\Gamma}} | r(z) - f(z)|$; update $\mathbf{S}$, $\mathbf{f}_{\mathbf{S}}$, $\mathbf{\Gamma}$ and $\mathbf{f}$ by adding $z_{k}$ to $\mathbf{S}$ and deleting  $z_{k}$ in $\mathbf{\Gamma}$, adding $f(z_{k})$ to $\mathbf{f}_{\mathbf{S}}$ and deleting it in $\mathbf{f}$.
\item Build the $(L-j-1) \times (j+1)$ matrices $\mathbf{C}_{j+1}\!\!:=\!\!\left( \frac{1}{z-k} \right)_{z \in \mathbf{\Gamma}, k \in \mathbf{S}}$, ${\mathbf A}_{j+1}\!\!:=\!\! \left( \frac{f(z) - f(k)}{z-k} \right)_{z \in \mathbf{\Gamma}, k \in \mathbf{S}}$.
\item Compute the singular vectors $
\mathbf{v}_1$ and $
\mathbf{v}_2$ corresponding to two smallest singular values of $\mathbf{A}_{j+1}$;
 compute $\mathbf{w}:=(\mathbf{v}_2^{T} \, \mathbf{f}_{\mathbf{S}})\mathbf{v}_1 - (\mathbf{v}_1^{T} \, \mathbf{f}_{\mathbf{S}})\mathbf{v}_2$ and normalize $\mathbf{w} :=\frac{1}{\|\mathbf{w}\|_2} \, \mathbf{w}$.
\item Compute $\mathbf{p}:=\mathbf{C}_{j+1} (\mathbf{w}.\ast \mathbf{f}_{\mathbf{S}})$, $\mathbf{q}:=\mathbf{C}_{j+1} \mathbf{w}$ and $\mathbf{r}:=(r(z))_{z\in \mathbf{\Gamma}}=\mathbf{p}./\mathbf{q} \in \cc^{L-1-j}$.
\item If $\| \mathbf{r}-\mathbf{f} \|_\infty<$ \textit{tol}  then stop.
\end{itemize}
end (for)\\
\textbf{Output: } \\ 
$N=j$ the order of the rational function $r_{N}$ \\
 $\mathbf{S} = (z_j)_{j=1}^{N+1} \in \cc^{N+1}$ is the vector of points with the interpolation property \\
  $\mathbf{f}_{\mathbf{S}}= (f(z_j))_{j=1}^{N+1} \in \cc^{N+1}$ is the vector of the corresponding interpolation function values\\
   $\mathbf{w}=(w_j)_{j=1}^{N+1} \in \cc^{N+1}$ is the weight vector. 

\end{algorithm}

Algorithm \ref{alg1} provides the rational function $r_N(z)$ in barycentric form 
$r_{N}(z)  = \frac{\tilde{p}_{N}(z)}{\tilde{q}_{N}(z)}$ with 
\begin{equation}\label{bar-form-1}
\tilde{p}_{N}(z) := \sum_{j=1}^{N+1} \frac{w_j \, f(z_j)}{z-z_j}, \qquad  \tilde{q}_{N}(z) := \sum_{j=1}^{N+1} \frac{w_j}{z-z_j},
\end{equation}
which are determined by the output parameters of this algorithm. 

\begin{remark}
Observe that the interpolation condition $r_N(z_j) = f(z_j)$ is only satisfied for the component $z_j$ of ${\mathbf S}$, if $w_j\neq 0$. If $w_j=0$ occurs, then $f(z_j)$ is a ``non-achievable'' value for this rational interpolation problem. We will use this property  of the modified AAA algorithm to detect and to reconstruct also $P$-periodic terms in exponential sums.
\end{remark}

In order to rewrite $r_N(z)$ in (\ref{bar-form-1}) in the form  of a partial fraction decomposition,
\begin{equation}\label{rn1} 
r_{N}(z) = \sum_{j=1}^{N} \frac{g_{j}}{z-\rho_{j}},
\end{equation}
we need to determine $g_1, \ldots , g_N$ and $\rho_1, \ldots , \rho_N$ from the output of Algorithm \ref{alg1}.

The zeros of denominator $\tilde{q}_{N}(z)$ are the poles $\rho_j$ of $r_N(z)$ and can be computed by solving an $(N+2)\times (N+2)$ generalized eigenvalue problem (see \cite{AAA} or \cite{PP2020}),
\begin{equation}\label{eig} \left( \begin{array}{ccccc}
0 & w_1 & w_2 & \ldots & w_{N+1} \\
1 & z_1 &   &     & \\
1 & & z_2 & & \\
\vdots & & & \ddots & \\
1 & & & & z_{N+1} \end{array} \right) \, {\mathbf v}_{\rho}= \rho \left( \begin{array}{ccccc}
0 & & & & \\
& 1 & & & \\
& & 1 & & \\
& & & \ddots & \\
& & & & 1 \end{array} \right) \, {\mathbf v}_{\rho}.
\end{equation}
Two eigenvalues of this generalized eigenvalue problem are infinite and the other $N$ eigenvalues are the wanted zeros $\rho_j$ of $\tilde{q}_{N}(z)$ (see \cite{PP2020} for more detailed explanation). 
We apply the following Algorithm \ref{alg2} to the output of Algorithm \ref{alg1}.

\begin{algorithm}[Reconstruction of parameters $g_j$ and $\rho_j$ of partial fraction representation]
 \label{alg2}
\textbf{Input: } $\mathbf{S}\in \zz^{N+1}$,
  ${\mathbf{f}}_{\mathbf{S}} \in \cc^{N+1}$,
   $\mathbf{w} \in \cc^{N+1}$ the output vectors of Algorithm \ref{alg1}.

\begin{itemize}
\item Build the matrices in (\ref{eig}) and solve this eigenvalue problem to find the  vector $\rho^{T}=(\rho_1, \ldots, \rho_N)^{T}$ of the $N$ finite eigenvalues;

\item Build the matrix $\mathbf{V}=\left(\frac{1}{z_{k}-\rho_j} \right)_{z_{k}\in \mathbf{S}, \, j=1,\ldots,N}\in \rr^{(N+1)\times N} $ and solve the linear system
$$
\mathbf{V} \mathbf{g}={\mathbf{f}}_{\mathbf{S}} .
$$
\end{itemize}
\textbf{Output: } Parameter vectors $\boldsymbol{ \rho}= (\rho_j)_{j=1}^{N}$, ${\mathbf g} = (g_j)_{j=1}^{N}$ determining $r_N$ in (\ref{rn1}). 
\end{algorithm}

\section{Recovery of Proper  Exponential Sums}
\label{proper-exp}

In this section we study the recovery of proper exponential sums  $y \in\mathcal{Y}_N^{0}$ of the form
\begin{equation}\label{func-def-4}
y(t)=\sum_{j=1}^N  \gamma_{j}  {\mathrm e}^{2\pi  \lambda_j t} , \quad  \gamma_{j} \in \cc\setminus \{0\}, \;  \lambda_j \in {\mathbb C},
\end{equation}
where $\lambda_{j}$ are assumed to be pairwise distinct. Note that we use frequencies $2\pi  \lambda_j $ instead of $ \lambda_j $ just for convenience.  We want to recover this exponential sum from a small set of Fourier coefficients of $y$ obtained from the Fourier series expansion on a finite interval $[0,P] \subset {\mathbb R}$ of length $P>0$.
First we study the special structure of  the Fourier coefficients  of $y(t)$.

\begin{lemma} \label{thpr2}
The function $\phi(t)=\gamma  {\mathrm e}^{2\pi \lambda t}$ with $\gamma \in \cc\setminus \{0\}$ and $\lambda \in {\mathbb C}$  can be expanded into a Fourier series   of the form
$
\phi(t)=\sum\limits_{k\in \zz} c_k(\phi) {\mathrm e}^{2\pi  \mathrm{i} \, k t/P}
$ 
on the finite interval $[0,P] \subset \rr$,
and the Fourier coefficients $c_k(\phi)$ for $k \in {\mathbb Z}$  are given by
\begin{equation}\label{coef-complex}
c_{k}(\phi) = \frac{1}{P} \int_{0}^{P} \phi(t) \, {\mathrm e}^{-2\pi {\mathrm i}kt/P} \, {\mathrm d} t = 
\begin{cases}
 \frac{\gamma  \, (1-{\mathrm e}^{2\pi \lambda P}) }{2\pi {\mathrm i}(k+{\mathrm i} \lambda P )}, & k \neq - {\mathrm i} \lambda P, \\
\gamma, & k = - {\mathrm i}\lambda P.
\end{cases}
\end{equation}

In particular, $c_{k}(\phi) \neq 0$ for all $k \in {\mathbb Z}$ if $ -{\mathrm i}  \lambda \not\in \frac{1}{P} {\mathbb Z}$, and
$c_{k}(\phi) = \gamma \, \delta_{k,\ell} $ for all $k \in {\mathbb Z}$ if $ -{\mathrm i} \lambda =\frac{\ell}{P}$ for some $\ell \in \zz$, where $\delta_{k,\ell}$ denotes the usual Kronecker symbol.
\end{lemma}
\begin{proof} Since the function $\phi(t)$ is differentiable on $\rr$, the Fourier coefficients of $\phi$ and the Fourier series expansion are well-defined (see, for example, \cite[Chapter 1]{Pbook}). 
A simple computation yields for $k \neq - {\mathrm i} \lambda P$
\begin{align*}
c_k(\phi) &= \frac{1}{P} \int_0^P \phi(t) \, {\mathrm e}^{-2\pi  \mathrm{i} \, k t/P} \, {\mathrm d} t  = \frac{\gamma}{P} \int_0^P \, {\mathrm e}^{2\pi \left(\lambda- \mathrm{i} \, k/P\right) t}  \, {\mathrm d} t \\
&= \frac{\gamma}{P} \left( \frac{{\mathrm e}^{2\pi(\lambda - {\mathrm i} k/P)P} - 1}{2\pi (\lambda-{\mathrm i} k/P)}\right) =    \frac{\gamma  \, (1-{\mathrm e}^{2\pi \lambda P}) }{2\pi {\mathrm i}(k+{\mathrm i} \lambda P )}.\\
\end{align*}
For $ k = - {\mathrm i}\lambda P$ we simply have 
$
c_k(\phi) =\frac{\gamma}{P} \int_0^P 1 \, {\mathrm d} t =\gamma.
$
\end{proof}

\subsection{Functions Containing only Non-$P$-periodic Terms}
Let us assume first that $-{\mathrm i}\lambda_{j} \not\in \frac{1}{P} {\mathbb Z}$ for $j=1, \ldots , N$ in (\ref{func-def-4}). 
We introduce 
$$
C_j:=-{\mathrm i} \, \lambda_j P, \qquad A_j:=  \frac{\gamma_j  \, (1-{\mathrm e}^{2\pi \lambda_j P}) }{2\pi {\mathrm i}} , \qquad j=1, \ldots, N.
$$
Then, the Fourier coefficients of $y(t)$ in  (\ref{func-def-4}) can by Lemma \ref{thpr2} be written as
\begin{equation}\label{tilde1}
c_k(y) = \sum \limits_{j=1}^{N} \frac{A_j}{k-C_{j}}.
\end{equation}
In other words, the sequence of Fourier coefficients $c_{k}(y)$ is already determined by a rational function 
\begin{equation}\label{tilde3}
r_{N}(z) :=\frac{p_{N-1}(z)}{q_N(z)}= \frac{\sum\limits_{j=1}^{N} A_j \prod\limits_{\mycom{s=1}{s \neq j}}^{N} (z-C_s)}{\prod\limits_{j=1}^{N}(z-C_j)}
\end{equation}
  of type $(N-1,N)$ satisfying $r_{N}(k) = c_{k}(y)$, and moreover, there is a bijection between the parameters sets $\gamma_{j}, \, \lambda_{j}$,  $j=1, \ldots , N$,  determining $y(t)$ in  (\ref{func-def-4}) and $A_{j}, \, C_{j}$, $j=1, \ldots, N$, determining $r_{N}(z)$ in (\ref{tilde3}), where
\begin{equation}\label{beta}
\lambda_{j} = \frac{{\mathrm i} C_{j}}{P}, \qquad 
\gamma_{j} =  \frac{2\pi {\mathrm i} A_{j} }{(1-{\mathrm e}^{2\pi \lambda_{j}P})} =  \frac{2\pi {\mathrm i} A_{j} }{(1-{\mathrm e}^{2\pi {\mathrm i}C_{j}})}.
\end{equation}

We obtain

\begin{theorem}\label{theo2}  Let $y$ be of the form $(\ref{func-def-4})$ with $N \in {\mathbb N}$ and $ \lambda_{j} \in {\mathbb C} \setminus \frac{{\mathrm i}}{P} {\mathbb Z}$ and  $\gamma_{j} \in {\mathbb C} \setminus \{0\}$, where we assume that $ \lambda_{j}$ are pairwise distinct. 
Let $\{c_{k}(y): \, k \in \Gamma\}$ with $\Gamma \subset {\mathbb Z}$ be a set of $L \ge 2N+1$  Fourier coefficients of the Fourier expansion of $y$ on the finite interval $[0,P] \subset {\mathbb R}$.
Then $y$ is uniquely determined by $2N$ of these Fourier
coefficients and Algorithm $\ref{alg1}$ $($with ${\mathbf \Gamma} = (k)_{k \in \Gamma}$ and ${\mathbf f} := (c_{k}(y))_{k \in \Gamma}$ with $c_{k}(y)$ in $(\ref{tilde1}))$  terminates after $N$ steps taking $N + 1$ interpolation points and provides a rational function $r_{N}(z)$ that satisfies  $c_{k}(y) = r_{N}(k)$ for all $k \in {\mathbb Z}$.
\end{theorem}
\begin{proof}
A rational function $r_{N}(z)= p_{N-1}(z)/q_{N}(z)$ with polynomials $p_{N-1}(z)$ of degree at most $N-1$ and $q_{N}(z)$ of degree exactly $N$, is already completely determined by $2N$ (independent) interpolations conditions $r_{N}(k) = c_{k}(y)$, if we can assume that the rational interpolation problem is solvable at all. But solvability can be assumed since we know that the coefficients $c_{k}(y)$ possess the structure given in (\ref{tilde3}). Linear independence of the conditions follows also from  (\ref{tilde3}), since the coefficients $c_{k}(y)$ cannot be presented by a rational function of smaller type than $(N-1, N)$.

Algorithm \ref{alg1} chooses at the $N$th step a set  $S_{N+1} \subset \Gamma$ of $N+1$ indices for interpolation such that $f(k) := c_{k}(y)$ for $k \in S_{N+1}$ (i.e. ${\mathbf S}= (k)_{k \in S_{N+1}}$ and ${\mathbf f}_{\mathbf S}= (c_{k}(y))_{k \in S_{N+1}}$ in Algorithm \ref{alg1}), and builds the matrix ${\mathbf A}_{N+1}= \left( \frac{c_{n}(y)-c_{k}(y)}{n-k} \right)_{n\in \Gamma \setminus S_{N+1},k \in S_{N+1}} \in {\mathbb C}^{L-N-1 \times N+1}$. Using the known structure of $c_{k}(y)$ in (\ref{tilde1}) we find the factorization
\begin{align*}
{\mathbf A}_{N+1} &= \left( \frac{\sum_{j=1}^{N} A_{j}  \Big(\frac{1}{n -C_{j}}- \frac{1}{k -C_{j}}\Big)}{n-k} \right)_{n\in \Gamma \setminus S_{N+1},k \in S_{N+1}}\\
&= \left( \frac{1}{n-C_{\ell}} \right)_{n \in \Gamma \setminus S_{N+1}, \ell=1, \ldots , N} \, \mathrm{diag} \, \Big( (-A_{\ell})_{\ell=1}^{N}\Big) \, \left( \frac{1}{k-C_{\ell}} \right)_{\ell=1, \ldots , N, k \in S_{N+1} }.
\end{align*}
The matrix ${\mathbf A}_{N+1}$ has exactly the rank $N$, since all three matrix factors have full rank $N$. Thus, there is a right (normalized) singular vector ${\mathbf v}_{1}$ of ${\mathbf A}_{N+1}$ to the singular value $\sigma_{1}=0$, i.e., ${\mathbf A}_{N+1} {\mathbf v}_{1} = {\mathbf 0}$, and the factorization above implies that also $ \left( \frac{1}{k-C_{\ell}} \right)_{\ell=1, \ldots , N, k \in S_{N+1} } {\mathbf v}_{1} = {\mathbf 0}$. Since any $N$ columns of  $ \left( \frac{1}{k-C_{\ell}} \right)_{\ell=1, \ldots , N, k \in S_{N+1} }$
are linearly independent, it follows that all components of ${\mathbf v}_{1}$ are nonzero. 
We conclude that the choice of a weight  vector $\tilde{\mathbf w} = {\mathbf v}_{1}$ ensures that the rational function $r_N(z)$ in barycentric form (\ref{bar-form-1}) constructed by Algorithm \ref{alg1} satisfies all interpolation conditions $r_{N}(k) = c_{k}(y)$ for $k \in S_{N+1}$ and moreover $r_{N}(n) = c_{n}(y)$ for $n \in \Gamma \setminus S_{N+1}$  by ${\mathbf A}_{N+1} \tilde
{\mathbf w} ={\mathbf 0}$. 
Since this rational  function $r_{N}(z)$ satisfies  $2N+1$ interpolation conditions, it follows that it coincides with the rational function $r_{N}(z)$ in (\ref{tilde3}).  But $r_{N}(z)$ in (\ref{tilde3}) is of type $(N-1,N)$, it follows that $\tilde{\mathbf w}$ also satisfies the second side condition of minimization problem (\ref{mini1}). Thus,    Algorithm \ref{alg1}  will provide the weight vector ${\mathbf w}= \tilde{\mathbf w}={\mathbf v}_{1}$ at the $N$th iteration step.
\end{proof}

The proof of Theorem \ref{theo2} implies that  in the considered case the kernel vector ${\mathbf w}= {\mathbf v}_{1}$ of the Loewner matrix ${\mathbf A}_{N+1}$  already satisfies  the two side conditions. Therefore, the modification of the original AAA-algorithm that ensures that the resulting rational approximant is of type $(N-1, N)$ is not needed in this case. 
The reconstruction of $y(t)$ in (\ref{func-def-4}) can now be summarized as follows.
\begin{algorithm}[Reconstruction of the parameters $\lambda_{j}, \, \gamma_{j}$  in  (\ref{func-def-4})] 
\label{alg:proper}
\textbf{Input: } ${\mathbf \Gamma} = (k)_{k \in \Gamma} \subset {\mathbb Z}$,  ${\mathbf f} := (c_{k}(y))_{k \in \Gamma}$ with $c_{k}(y)$ in $(\ref{tilde1}))$,  
\begin{description}
\item {1)} Apply Algorithm \ref{alg1} with ${\mathbf \Gamma} = (k)_{k \in \Gamma}$ and ${\mathbf f} := (c_{k}(y))_{k \in \Gamma}$ with $c_{k}(y)$, $tol = 10^{-13}$
in $(\ref{tilde1})$. We obtain ${\mathbf S} \in {\mathbb Z}^{N+1}$, ${\mathbf f}_{S} = (c_{k}(y))_{k \in {\mathbf S}}$ and ${\mathbf w} \in {\mathbb C}^{N+1}$.
\item{2)} Apply Algorithm \ref{alg2} to obtain $A_{j} = g_{j}$ and $C_{j} = \rho_{j}$, $j=1, \ldots, N$. 
\item{3)} Apply (\ref{beta}) to compute $\lambda_{j}, \, \gamma_{j}$, $j=1, \ldots , N$.
\end{description}
\textbf{Output: } $\lambda_{j}, \, \gamma_{j}$, $j=1, \ldots , N$, determining $y(t)$ in (\ref{func-def-4}).
\end{algorithm}

\subsection{Functions Containing also $P$-periodic Terms} 
\label{sec:periodic}
Now, we assume that the exponential sum (\ref{func-def-4}) contains beside non-$P$-periodic terms  $y_j(t)=\gamma_j \, {\mathrm e}^{2\pi  \lambda_j t}$ with $-{\mathrm i}\lambda_j \not \in \frac{1}{P} \zz$  also $P$-periodic terms with $-{\mathrm i} \lambda_j \in \frac{1}{P} \zz$. As seen in Lemma \ref{thpr2}, each $P$-periodic term provides only one non-zero Fourier coefficient, i.e.,  
$c_{k}(\gamma_j \, {\mathrm e}^{2\pi  \lambda_j t}) = \gamma_{j} \delta_{k,-{\mathrm i} \lambda_{j}P}$. 
Therefore, we 
assume that the index set $\Gamma$ of given Fourier coefficients $c_k(y)$ contains all integers $\{-{\mathrm i} \lambda_{j} P: j=1, \ldots , N\} \cap {\mathbb Z}$.
If $c_{k}(y)$ with $k=-{\mathrm i} \lambda_{j} P$ is not provided, then the term $y_j(t)= \gamma_{j} {\mathrm e}^{2\pi \lambda_{j}t}$ cannot be identified from the given data. 

Now the function $y$ in (\ref{func-def-4}) can be written as $y=y^{(1)}+y^{(2)}$, where 
\begin{equation}\label{y12}
y^{(1)}(t):=\sum_{j=1}^{N_1}  \gamma_{j} \,  {\mathrm e}^{2\pi \lambda_j t} \; \mathrm{with} \;    -{\mathrm i}\lambda_j \not \in \frac{1}{P} {\mathbb Z}, \quad y^{(2)}(t):=\sum_{j=N_1+1}^{N} \!\!\! \gamma_{j} \,  {\mathrm e}^{2\pi \lambda_j t} \; \mathrm{with} \;   -\mathrm{i} \lambda_j  \in \frac{1}{P} {\mathbb Z},
\end{equation}  
and $N_1<N$. The part $y^{(1)}(t)$ is non-$P$-periodic with the Fourier coefficients
$$
c_k(y^{(1)})= \sum \limits_{j=1}^{N_1} \frac{A_j}{k-C_j }=\frac{p_{N_1-1}(k)}{q_{N_1}(k)}=r_{N_1}(k), \qquad k \in {\mathbb Z},
$$
and the part $y^{(2)}(t)$ is $P$-periodic. We denote $\Sigma:=\{ -{\mathrm i} \lambda_j P: \, j=N_{1}+1, \ldots , N\}$. Then 
$$
c_k(y^{(2)})=
\begin{cases}
\gamma_j, & k \in \Sigma,\\
0, & k \not \in \Sigma.
\end{cases}
$$
The reconstruction of $y(t)$ is now based on the observation that all but $c_{k}(y)$, $k \in \Sigma$, still have the structure of  a rational function $r_{N_{1}}$, i.e., $c_{k}(y) = r_{N_{1}}(k)$ for $k \in {\mathbb Z}\setminus \Sigma$,  and $y^{(1)}(t)$ can be reconstructed by Algorithm \ref{alg1} while the $P$-periodic part $y^{(2)}(t)$ can be determined in a post-processing step. 

\begin{theorem} \label{theoper}
Let y in $(\ref{func-def-4})$ be of the form $y(t) = y^{(1)}(t) + y^{(2)}(t)$ as in $(\ref{y12})$, where we assume that $ \lambda_{j}$ are pairwise distinct. 
Let $\{c_{k}(y): \, k \in \Gamma\}$ with $\Gamma \subset {\mathbb Z}$ be a set of $L \ge 2N+2$  Fourier coefficients of the Fourier expansion of $y$ on the finite interval $[0,P] \subset {\mathbb R}$ with $P>0$, and assume that the (unknown) index set $\Sigma$ of non-zero Fourier coefficients of $y^{(2)}(t)$ is a subset of $\Gamma$.
Then $y^{(1)}$ and $y^{(2)}$ can be uniquely recovered from this set of Fourier coefficients. Algorithm $\ref{alg1}$ $($with ${\mathbf \Gamma} = (k)_{k \in \Gamma}$ and ${\mathbf f} := (c_{k}(y))_{k \in \Gamma}$ $)$ terminates after at most $N+1$ steps  and provides a rational function $r_{N_{1}}(z)$ of type $(N_{1}-1,N_{1})$ that satisfies $c_{k}(y) = r_{N_{1}}(k)$ for all $k \in {\mathbb Z}\setminus \Sigma$.
\end{theorem}
\begin{proof}
The proof employs similar ideas as the proof of Theorem 5.1 in \cite{PP2020} despite the different context. We therefore only sketch the main ideas of the proof.

At the $(N+1)$-th  iteration step, Algorithm \ref{alg1}  has chosen a set 
 $S \subset \Gamma$ of $N+2$ indices used for interpolation.
Let $\Gamma_{S} := \Gamma \setminus S$. Then we obtain the matrix 
$${\mathbf A}_{N+2} = \left( \frac{c_{\ell}(y)
- c_{k}(y)}{\ell-k} \right)_{\ell \in \Gamma_{S}, k \in S} \in {\mathbb C}^{L-N-2 \times N+2}. 
$$
We show that ${\mathbf A}_{N+2}$ has at most rank $N$. Let $s$ be the number of indices of $\Sigma$ being contained in $S$.
We consider the partial matrix ${\mathbf A}_{11}$ of  ${\mathbf A}_{N+2}$ obtained by deleting the rows and columns corresponding to indices in $\Sigma$. 
Then, ${\mathbf A}_{11}$ has at least $N+2-s \ge N_{1}+2$ columns and, similarly as in the proof of Theorem \ref{theo1}, it follows from a factorization argument that ${\mathbf A}_{11}$ has rank $N_{1}$.
Therefore, the submatrix of ${\mathbf A}_{N+2}$ built with the columns of ${\mathbf A}_{N+2}$ that correspond to the $N+2-s$ columns of ${\mathbf A}_{11}$ has at most rank $N_{1} + (N-N_{1}-s)=N-s$, since there are at most $N-N_{1}-s$ rows of ${\mathbf A}_{N+2}$ which are deleted in ${\mathbf A}_{11}$. We conclude that 
the full matrix ${\mathbf A}_{N+2}$ has at most rank $N$. 

Thus, Algorithm \ref{alg1} finds two vectors ${\mathbf v}_{1}$ and ${\mathbf v}_{2}$ with 
${\mathbf A}_{N+2} {\mathbf v}_{1} = {\mathbf A}_{N+2}  {\mathbf v}_{2} = {\mathbf 0}$ and thus, there is a weight vector ${\mathbf w}$ with ${\mathbf A}_{N+2} {\mathbf w}= {\mathbf 0}$ satisfying the side conditions $(c_{k})_{k \in S}^{T} {\mathbf w} = 0$ and $\|{\mathbf w}\|_{2}=1$. 
Observe that $c_{k}(y)= c_{k}(y^{(1)})$ for all $k \in \Gamma \setminus \Sigma$. 
Since any  $N_{1}$ columns out of the columns of ${\mathbf A}_{11}$ are linearly independent, 
${\mathbf w}$ contains at least $N_{1}+1$ nonzero components corresponding to columns of ${\mathbf A}_{11}$ and therefore, the rational function $r_{N_{1}}$ 
obtained by Algorithm \ref{alg1} indeed interpolates all Fourier coefficients of $y^{(1)}$. 
Therefore, Algorithm \ref{alg1} determines $y^{(1)}$. In a post processing  step we find  all nonzero Fourier coefficients of $y^{(2)}$  by inspecting $c_{k}(y^{(2)}) = c_{k}(y) - r_{N_{1}}(k)$, $k \in \Gamma$, and can determine $y^{(2)}$.
 \end{proof}
 
\begin{remark} 
 Comparing Theorems \ref{theo2} and \ref{theoper} we see that the reconstruction procedure may require $N+1$ instead of $N$ steps
if the exponential sum also contains $P$-periodic components, where $N$ is the order of the corresponding exponential sum. 
In practice, Algorithm \ref{alg1} often terminates after $N$ steps, even if $P$-periodic terms appear.  
The Fourier coefficients that are deteriorated by the P-periodic part can however also produce a Froissart doublet.
Actually, if Algorithm \ref{alg1} stops, then all indices of Fourier coefficients that corresponds to the $P$-periodic components, 
have been taken into the set ${S}$, i.e., $\Sigma \subset S$. Otherwise, these Fourier coefficients would cause an error in the approximation step, since they do not have  the wanted rational structure. 
\end{remark}

To reconstruct the function $y(t)$ in (\ref{func-def-4}), we can again just apply Algorithm \ref{alg:proper} to reconstruct all parameters of the non-periodic part $y^{(1)}(t)$. The set $\Sigma$ can be found by determining all integer poles $C_j$ that are found in the second step of Algorithm \ref{alg:proper}.  The coefficients $\gamma_j$ for $j=N_1+1,...,N$ can be now reconstructed via 
\begin{equation} \label{coef-rec-prop}
\gamma_j=c_{k_j}(y)-c_{k_j}(y^{(1)}).
\end{equation}

\begin{example}
We consider the following proper exponential sum, see Figure \ref{fig2},
\begin{align} \nonumber 
y_{1}(t):=&(3.2+4.5 \mathrm{i})\,  {\mathrm e}^{2\pi \cdot (-1.095+\sqrt{0.0101} \,  \mathrm{i})t} -0.55 \, {\mathrm e}^{-2\pi \cdot 2.647 \mathrm{i} t} + (-3.4+0.1 \mathrm{i} ) \,  {\mathrm e}^{2\pi \cdot 1.3711 \mathrm{i} t} \\
& -0.88 \, {\mathrm e}^{-2\pi \cdot \sqrt{1.89}  t} + (0.542+7.1  \mathrm{i} ) \, {\mathrm e}^{2\pi  \cdot ( -\sqrt{0.47}+3.217 \mathrm{i} )t} +(-0.96+1.06  \mathrm{i}) \, {\mathrm e}^{-2\pi \cdot 2 \mathrm{i}} ,\label{y2}
\end{align}
i.e., $y_{1}(t)$ is of the form (\ref{func-def-4})  with $N=6$ and  the parameter vectors
\begin{align*}
\blambda&=(-1.095+\sqrt{0.0101} \,  \mathrm{i}, \, -2.647 \, \mathrm{i}, \,  1.3711 \, \mathrm{i},\, -  \sqrt{1.89}, \, -\sqrt{0.47}+3.217 \, \mathrm{i}, \, -2 \,  \mathrm{i}),\\
 \gamra &=(3.2+4.5 \mathrm{i}, \, -0.55,\,  -3.4+0.1 \mathrm{i}, \,  -0.88, \, 0.542+7.1\mathrm{i}, \, -0.96+1.06  \mathrm{i}).
\end{align*}

\begin{figure}[h]
\centering
\begin{subfigure}{0.46\textwidth}
\includegraphics[width=\textwidth]{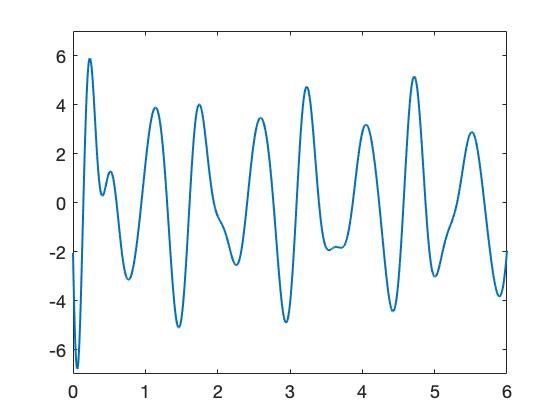}   
\end{subfigure}
\begin{subfigure}{0.46\textwidth}
\includegraphics[width=\textwidth]{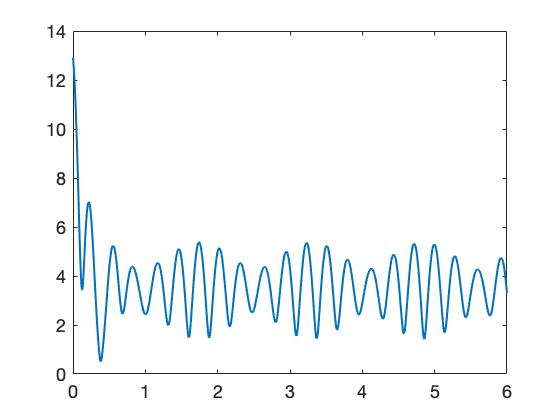}   
\end{subfigure}
\caption{Graph of $\mathrm{Re} \, y_{1}(t)$ (left) and  of $|y_{1}(t)|$ (right) for $y_{1}(t)$ in (\ref{y2}) on  $[0,6]$. }
\label{fig2}
\end{figure}

\noindent
We use  $P=6$, and employ the $59$ Fourier coefficients $c_k(y_{1})$, $k=-29,\ldots,29$, to recover $y_{1}(t)$. The function $y_{1}(t) = y_{1}^{(1)}(t) + y_{1}^{(2)}(t)$ contains one $6$-periodic term and 5 non-$6$-periodic terms,
\begin{align*} 
y_{1}^{(2)}(t) &= (-0.96+1.06  \mathrm{i}) \, {\mathrm e}^{-2\pi \cdot 2 \mathrm{i}} , \\
y_{1}^{(1)}(t) &= y_{1}(t) - y_{1}^{(2)}(t).
\end{align*} 
Algorithm \ref{alg1} iteratively employs the 7 Fourier coefficients $c_{8}(y_{1})$, $c_{-12}(y_{1})$,  $c_{9}(y_{1})$, $c_{-13}(y_{1})$,  $c_{-16}(y_{1})$, $c_{20}(y_{1})$ and $c_{0}(y_{1})$ (in this order) for interpolation before it stops  after 6 iteration steps with error $3.6 \cdot 10^{-16}$.
Here, the  Fourier coefficient $c_{-12}(y_{1})$, 
which contains information about $y_{1}^{(2)}$, is already taken. The obtained rational function $r(z)$  is already completely determined by the remaining $6$ Fourier coefficients 
$c_{k}(y_{1}) = c_{k}(y_{1}^{(1)})$, $k=8, \, 9, \, -13, \, -16, \, 20, \, 0$. Therefore, Algorithms \ref{alg1} provides $r(z)$ that interpolates $c_{k}(y_{1}^{(1)})$ for all $k$, while it does not interpolate $c_{k}(y_{1}) \neq c_{k}(y_{1}^{(1)})$ for $k=-12$. Indeed we observe that 
the second component of ${\mathbf  w} \in \cc^{7}$ vanishes, indicating that $c_{-12}(y)$  is a ``non-achievable'' point for this rational interpolation. 
After skipping this vanishing term in ${\mathbf w}$ and in ${\mathbf S}$ correspondingly, we obtain a rational function $r_5(z)$ in the barycentric form of type $(4,5)$ that is determined by
$$
 {\mathbf  S}=\left(
\begin{matrix}
  8\\
     9\\
   -13\\
   -16\\
    20\\
     0
\end{matrix}
\right), \, \, \, \, \, \, \, \, \, 
 {\mathbf  w}=\left(
\begin{matrix}
  -0.096270855302241 - 0.103263179235592 \mathrm{i} \\
 -0.265355731611110 - 0.389226329639419 \mathrm{i} \\
  0.719319960483190 - 0.000900124596615 \mathrm{i} \\
  0.028635901429394 - 0.005047069650253 \mathrm{i} \\
 -0.228142957097121 + 0.028004149922005 \mathrm{i} \\
  0.224577809992588 - 0.369623728522123 \mathrm{i} 
\end{matrix}
\right),
$$
and the vector of Fourier coefficients with indices corresponding to the index vector ${\mathbf S}$.
To reconstruct the non-periodic part $y_{1}^{(1)}(t)$ of $y_{}(t)$, we apply Algorithm \ref{alg:proper}  as described in the previous subsection. 
Finally, we reconstruct the periodic part $y_{1}^{(2)}(t)$.  Comparing the Fourier coefficients with values of  $r_5(z)$ we find the set $\Sigma= \{-12\}$.  According to  (\ref{coef-rec-prop}), $c_{-12}(y)$ already completely covers $y_{1}^{(2)}(t)$. 
 The obtained reconstruction errors are
$$
\| \tilde{\blambda} -\blambda\|_\infty=1.72 \cdot 10^{-12}, \qquad \| \tilde{\gamra} -\gamra\|_\infty=1.69 \cdot 10^{-11},
$$
where $\tilde{\blambda}$ and $\tilde{\gamra}$ denote the reconstructed parameter vectors.
\end{example}

\subsection{Recovery of Real Proper Exponential Sums with Real Frequencies}
In this section, we consider the recovery of real proper exponential sums 
\begin{equation}\label{func-def-2}
y(t)=\sum_{j=1}^N  \gamma_{j}  {\mathrm e}^{2\pi \alpha_j t} , \quad \gamma_{j},\, \alpha_j  \in {\mathbb R}\setminus \{0\}
\end{equation}
from a small number of its Fourier coefficients obtained for the Fourier series expansion of $y$ on a given finite interval $[0,P]$.
In this case, we can derive  a special algorithm in real arithmetic.
We start with studying the structure of the Fourier coefficients of $y$. 
\begin{lemma} \label{thpr1}
The function $\phi(t)=\gamma  {\mathrm e}^{2\pi \alpha t}$ with $\gamma, \, \alpha \in {\mathbb R}\setminus \{0\}$  can be expanded into the Fourier series on the finite interval $[0,P] \subset \rr$ with $P >0$ of the form
$
\phi(t)=\sum\limits_{k\in \zz} c_k(\phi) {\mathrm e}^{2\pi  \mathrm{i} \, k t/P},
$ 
and the Fourier coefficients $c_k(\phi)$ for $k \in {\mathbb Z}$  are given by
$$
c_{k}(\phi) = \frac{1}{P} \int_{0}^{P} \phi(t) \, {\mathrm e}^{-2\pi {\mathrm i}kt/P} \, {\mathrm d} t = \frac{\gamma \,   {\mathrm e}^{\pi \alpha P} \, (P\alpha + {\mathrm i} k)\, \sinh(\pi \alpha P) }{\pi (\alpha^{2} P^{2} + k^{2})}.
$$
\end{lemma}
\begin{proof}
The proof follows the same lines as the proof of Lemma \ref{thpr2}.
\end{proof}

\noindent
Now, the function $y$ in (\ref{func-def-2}) can be rewritten  as
$y(t)=\sum_{j=1}^N  y_j(t)$ with $y_j(t):=\gamma_j  {\mathrm e}^{2\pi \alpha_j t}$,
and, according to Lemma \ref{thpr1},  its Fourier coefficients  satisfy the representation
\begin{equation}\label{coef-cky}
c_k(y)= \sum_{j=1}^N  \frac{A_j + {\mathrm i}k  B_{j}}{C_j+k^{2}}
\end{equation}
with real parameters
\begin{align*}
C_j &:=\alpha_j^{2} P^{2}, \nonumber \\ 
A_j &:=\frac{\gamma_j P \alpha_{j}}{\pi}  {\mathrm e}^{\alpha_{j} \pi P} \, \sinh(\pi \alpha_{j} P)    \\ 
B_j &:= \frac{\gamma_j}{\pi}  {\mathrm e}^{\alpha_{j} \pi P} \, \sinh(\pi \alpha_{j} P), 
\end{align*}
 for all $j=1,\ldots,N$. Conversely, the coefficients $A_j$ and  $B_{j}$,  $j=1,\ldots,N$, uniquely determine the parameters $\alpha_{j}$ and $\gamma_{j}$ of the exponential sum (\ref{func-def-2}). We have 
\begin{equation}\label{alpha}
\alpha_j=\frac{A_{j}}{B_{j}P}
\end{equation}
and 
\begin{equation}\label{gamma1}
\gamma_j= \frac{B_{j}\pi}{{\mathrm e}^{\alpha_{j}\pi P} \sinh(\pi \alpha_{j} P)} =\frac{B_{j}\pi}{{\mathrm e}^{A_{j}\pi/B_{j}} \sinh(\pi A_{j}/ B_{j})}.
\end{equation}
Obviously, we also have $\alpha_j^2 = C_j/P^2$, and $C_j$ can hence be written as $C_j=\left( \frac{A_j}{B_j} \right)^2$.

We consider now the following modification of  Fourier coefficients $c_{k}(y)$ in (\ref{coef-cky}), 
\begin{equation}\label{tilde}
\tilde{c}_{k}(y) := \mathrm{Re} \, c_k(y)  + \frac{{\mathrm i}}{k}   \mathrm{Im} \, c_k(y) = \sum_{j=1}^{N} \frac{A_{j} + {\mathrm i} B_{j}}{k^{2} +C_{j}},
\end{equation}
which can be seen as the sample values of a rational function 
$$r_{N}(z) = \frac{p_{N-1}(z) }{ q_{N}(z)} = \sum_{j=1}^{N} \frac{A_{j} + {\mathrm i} B_{j}}{z +C_{j}} $$
 of type $(N-1,N)$ at $z=k^{2}$.
The reconstruction algorithm is now based on this observation. We obtain

\begin{theorem}\label{theo1}  Let $y$ be of the form $(\ref{func-def-2})$ with $N \in {\mathbb N}$ and $ \alpha_{j}, \, \gamma_{j} \in {\mathbb R} \setminus \{0\}$, where we assume that $ \alpha_{j}$ are pairwise distinct. 
Let $\{c_{k}(y): \, k \in \Gamma\}$,  with $\Gamma \subset {\mathbb N}$ be a set of $L \ge 2N+1$  Fourier coefficients of the Fourier expansion of $y$ on the finite interval $[0,P] \subset {\mathbb R}$.
Then $y$ is uniquely determined by $2N$ of these Fourier
coefficients and Algorithm $\ref{alg1}$ (with ${\mathbf \Gamma} = (k^{2})_{k \in \Gamma}$ and ${\mathbf f} := (\tilde{c}_{k}(y))_{k \in \Gamma}$ with $\tilde{c}_{k}(y)$ in $(\ref{tilde})$ terminates after N steps taking $N + 1$ interpolation points and provides a rational function $r_{N}(z)$ satisfying $r_N(k^2) = \tilde{c}_{k}(y)$ for all $k \in {\mathbb Z}$.
\end{theorem}

The proof of Theorem \ref{theo1} can be derived analogously as for Theorem \ref{theo2}.  In particular, Algorithm \ref{alg1} provides a reconstruction algorithm for the rational function $r_{N}(z)$ that determines the modified  Fourier coefficients of $y$.
For reconstruction of the parameters of $y(t)$ in (\ref{func-def-2}), we thus need again to proceed with the following steps.

\begin{algorithm}[Reconstruction of the parameters $\alpha_j$ and $\gamma_j$ from (\ref{func-def-2})]
\label{alg:proper2}
\textbf{Input: } ${\mathbf \Gamma} = (k^{2})_{k \in \Gamma} \subset {\mathbb Z}$,  ${\mathbf f} := (\tilde{c}_{k}(y))_{k \in \Gamma}$ with $\tilde{c}_{k}(y)$ in $(\ref{tilde}))$,  
\begin{description}
\item{1)}
Apply Algorithm \ref{alg1} to  compute a rational function $r_{N}(z)$ of type $(N-1, N)$ from a set of $L \ge 2N+1$ Fourier coefficients with nonnegative index. Use the input data ${\mathbf \Gamma} = (k^{2})_{k \in \Gamma}$ and ${\mathbf f} := (\tilde{c}_{k}(y))_{k \in \Gamma}$ with $\tilde{c}_{k}(y)$ in $(\ref{tilde})$.
Algorithm \ref{alg1} then provides the vector of used (squared) indices ${\mathbf S}=(z_{j})_{j=1}^{N+1}$ with $z_{j} = k_{j}^{2}$, $k_{j} \in \Gamma \subset {\mathbb N}$, the vector of used modified Fourier coefficients ${\mathbf f}_{{\mathbf S}} = (\tilde{c}_{k_{j}}(y))_{j=1}^{N+1}$ and the weight vector ${\mathbf w}=(w_{j})_{j=1}^{N+1}$ to determine $r_{N}(z)$ of the form (\ref{bar-form-1}). 
\item{2)} Rewrite $r_{N}$ in the form 
$$
r_N(z)= \frac{\tilde{p}_{N}(z)}{\tilde{q}_{N}(z)} = \frac{\sum_{j=1}^{N+1} \frac{w_j \, \tilde{c}_{k_{j}}(y)}{z-k_{j}^{2}}}{\sum_{j=1}^{N+1} \frac{w_j}{z-k_{j}^{2}}} =
\sum \limits_{j=1}^{N} \frac{A_j+\mathrm{i} \, B_j}{z+C_j},
$$
i.e., extract $A_{j}, \, B_{j}, C_{j}$ from ${\mathbf S}, \, {\mathbf f}_{{\mathbf S}}$, and ${\mathbf w}$.
This is done by employing Algorithm \ref{alg2} with the output $g_{j}=A_{j} + {\mathrm i} B_{j}$ and $\rho_{j} = C_{j}$.
\item{3)} 
Compute the parameters $\alpha_{j}$, $\gamma_{j}$ from $A_{j}$ and $B_{j}$, $j=1, \ldots , N$, via (\ref{alpha}) and (\ref{gamma1}).
\end{description}
\textbf{Output: } $\alpha_{j}$, $\gamma_{j}$, $j=1, \ldots , N$, determining $y(t)$ in (\ref{func-def-2}). 
\end{algorithm}

\begin{example}
We consider the following exponential sum, see Fig.\ \ref{fig1},
\begin{align*}
y_2(t) := & -0.00572 \, {\mathrm e}^{-6.74 \cdot 2\pi t}+0.1074 \, {\mathrm e}^{-3.187  \cdot 2\pi  t}-0.685 \,  {\mathrm e}^{-1.312 \cdot 2\pi   t}-0.4264 \,  {\mathrm e}^{-1.212  \cdot 2\pi  t} \\
& +0.4605 \,  {\mathrm e}^{0.223 \cdot 2\pi  t},
\end{align*}
i.e., $y_{2}(t)$ is of the form (\ref{func-def-2}) with $N=5$ and parameter vectors
\begin{align*}
 {\balpha} &=(-6.74, \, -3.187, \, -1.312, \, -1.212, \, 0.223),  \\  \gamra &=(-0.00572, \, 0.1074, \, -0.685, \, -0.4264, \, 0.4605).
\end{align*}

\begin{figure}[h]
\includegraphics[width=8cm]{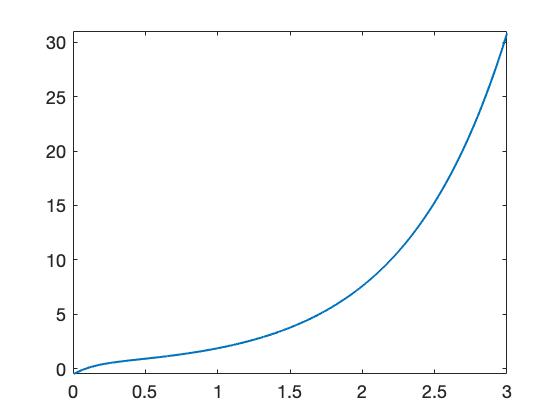}
\centering
\caption{ Graph of the proper exponential sum $y_2$ on  $[0,3]$. }
\label{fig1}
\end{figure}

For the recovery of $y_2$ we choose $P=3$, and employ 40 Fourier coefficients $\tilde{c}_k(y_2)$, $k=1,\ldots, 40$.  Algorithm \ref{alg1} iteratively employs the $N+1$ coefficients $\tilde{c}_1(y_2)$, $\tilde{c}_2(y_2)$, $\tilde{c}_4(y_2)$, $\tilde{c}_{40}(y_2)$, $\tilde{c}_{15}(y_2)$, and $\tilde{c}_{27}(y_2)$ for interpolation (in this order), before it terminates with the error $1.6 \cdot 10^{-16}$.  We get a rational function $r_{5}(z)$ of type (4, 5) that is determined by 
$$
 {\mathbf  S}=\left(
\begin{matrix}
1^{2} \\
2^{2} \\
4^{2} \\
40^{2} \\
15^{2} \\
27^{2}
\end{matrix}
\right), \, \, \, \, \, \, \, \, \, 
 {\mathbf  w}=\left(
\begin{matrix}
 0.000107694973344 - 0.000273615982469\mathrm{i }\\
-0.000627784609123 + 0.001594985329941\mathrm{i }\ \\
0.001597607891293 - 0.004058973590996\mathrm{i }\ \\
-0.287695273175763 + 0.730934994893030\mathrm{i }\ \\
-0.035505720375047 + 0.090207855189328\mathrm{i }\ \\
0.223846148680396 - 0.568716273077864\mathrm{i }\
\end{matrix}
\right)
$$
and ${\mathbf f}_{S} = (\tilde{c}_{k}(y_{2}))_{k^{2} \in {\mathbf S}}$.
We reconstruct parameters  $\tilde{\alpha}_{j}$  and  $\tilde{\gamma}_{j}$, $j=1, \ldots , 5$,  by Algorithm \ref{alg:proper2} with the errors
$$
\| {\balpha} - \tilde{  {\balpha}}\|_{\infty}=5.52 \cdot 10^{-11}, \ \ \ \|\gamra - \tilde{\gamra }\|_{\infty}=2.16 \cdot 10^{-10}.
$$
\end{example}

\section{Recovery of  Extended Exponential Sums}
\label{extended-exp}

In this section we study the recovery of 
extended exponential sums $y \in\mathcal{Y}_N$, 
\begin{equation}\label{func-def-5}
y(t)=\sum_{j=1}^M \left( \sum_{m=0}^{n_j}  \, \gamma_{j,m} \, t^{m} \right) {\mathrm e}^{2 \pi  \lambda_jt}, \qquad  \gamma_{j,m} \in \cc, \, \gamma_{j, n_{j}}\neq 0, \, \lambda_j \in \cc,
\end{equation}
of full order $N:=\sum\limits_{j=1}^{M}(1+n_j)$ from a small set of given Fourier coefficients  obtained  for the Fourier series expansion of $y$ on $[0, P]$. Note  that for convenience we consider again frequencies $2\pi \lambda_j$ instead of $\lambda_j$ for $j=1,\ldots ,M$.

\subsection{Representation of Fourier Coefficients via Rational Functions}
\label{sec5.1}

First we study the structure of Fourier coefficients  of functions  $y(t)$ in (\ref{func-def-5}).
We start with  the following result regarding the expansion of one component in (\ref{func-def-5}) into a Fourier series  on $[0, P]$ 
with some given $P>0$.
\begin{theorem}\label{thpr3}
The function $\phi(t)=\gamma \, t^{m} \, {\mathrm e}^{2\pi \lambda t}$ with $m\in \nn$, $\lambda \in \cc$, and  $\gamma \in \cc \setminus \{ 0 \}$ can be expanded on $[0,P]$ into the Fourier series 
$
\phi(t)=\sum\limits_{k\in \zz} c_k(\phi) \, {\mathrm e}^{2\pi \mathrm{i} k t/P}.
$
If $-{\mathrm i} \lambda \not \in \frac{1}{P} {\mathbb Z}$, the Fourier coefficients $c_k(\phi)$, $k \in {\mathbb Z}$, are given by
\begin{equation}\label{coef-pr2}
c_k(\phi)  = \frac{\gamma \, P^m m!}{(2\pi {\mathrm i})^{m+1} \, (k+ {\mathrm i} \lambda P)^{m+1}} \left( 1 - 
{\mathrm e}^{2\pi \lambda P}  \sum\limits_{\ell=0}^{m} \frac{1}{\ell!} \, (2\pi {\mathrm i})^{\ell} \, (k + {\mathrm i} \lambda P)^{\ell}\right).
\end{equation}
If $-{\mathrm i} \lambda  \in \frac{1}{P} \zz$, i.e., if there exists an $n \in \zz$ with $-{\mathrm i} \lambda P = n$, then
\begin{equation}\label{coef-pr2-per}
c_k(\phi)=
\begin{cases}
\frac{\gamma P^{m}}{m+1}, & k=n, \\
- \frac{\gamma \, P^m  \, m!}{(2\pi {\mathrm i})^{m+1} \, (k-n)^{m+1}}\, \left( \sum\limits_{\ell=1}^m \frac{1}{\ell !} \, (2\pi {\mathrm i})^{\ell} \, (k- n)^{\ell}\right), 
& k \in {\mathbb Z} \setminus\{n\}.
\end{cases}
\end{equation}
\end{theorem}
\begin{proof}
Since $y(t)$ in (\ref{func-def-5}) is differentiable on $[0, P]$, the Fourier series and all Fourier coefficients are well defined, and we have pointwise convergence in $(0, P)$, see \cite{Pbook}.
Let first $-{\mathrm i}\lambda P \not \in {\mathbb Z}$.  Applying \cite[2.321]{GR}, for $a \neq 0$ the general antiderivative of $t^{m} {\mathrm e}^{a t}$ reads
$$
\int t^{m} {\mathrm e}^{a t} \, dt={\mathrm e}^{a t} \, \left( \sum\limits_{\ell=0}^{m} \frac{(-1)^{\ell} \, \ell! \, \binom{m}{\ell}}{a^{\ell+1}} \, t^{m-\ell}  \right).
$$
Thus, for $k \in {\mathbb Z}$, 
\begin{align*}
 c_k(\phi) & =  \frac{\gamma}{P} \int_0^P \, t^{m} {\mathrm e}^{2\pi  \left( \lambda - \mathrm{i}k/P\right) t}  \, {\mathrm d} t \notag  \\
 & =  \frac{\gamma}{P} {\mathrm e}^{2\pi  \left(\lambda- \mathrm{i}k/P\right)P}  \sum\limits_{\ell=0}^m \frac{(-1)^\ell \ell ! \binom{m}{\ell}}{\left(2\pi  \left(\lambda-\mathrm{i}k/P\right) \right)^{\ell+1}} P^{m-\ell}  - \frac{\gamma}{P}\cdot\frac{(-1)^m m!}{\left(2\pi  \left(\lambda- \mathrm{i}k/P\right) \right)^{m+1}} \\
 & = \frac{\gamma \, P^m m!}{(2\pi {\mathrm i})^{m+1} \, (k + \mathrm{i} \lambda P)^{m+1}} \left( 1 - 
{\mathrm e}^{2\pi \lambda P}  \sum\limits_{\ell=0}^{m} \frac{1}{\ell!} \, (2\pi {\mathrm i})^{\ell} \, (k+ \mathrm{i} \lambda P)^{\ell}\right).
 \end{align*}
Let now $-{\mathrm i} \lambda P \in \zz$. Then there is an $n \in \zz$ such that $- {\mathrm i} \lambda P=n$, and we obtain
$
c_n(\phi)=\frac{\gamma}{P}\int\limits_0^{P} t^{m} \, dt=\frac{\gamma P^{m}}{m+1}.
$
For $k \in {\mathbb Z} \setminus \{n\}$ we have
\begin{align*}
c_k(\phi) & =\frac{\gamma}{P} \int_0^P \, t^{m} {\mathrm e}^{2\pi \mathrm{i} \left(n-k\right) t/P}  \, {\mathrm d} t = \gamma P^{m} \sum\limits_{\ell=0}^{m-1} \frac{(-1)^{\ell} \, \ell! \, \binom{m}{\ell}}{(2\pi \mathrm{i}(n-k))^{\ell+1}} \notag \\
& = - \gamma P^{m} \sum\limits_{\ell=0}^{m-1} \frac{{\mathrm i}^{\ell+1} \, \ell! \, \binom{m}{\ell}}{(2\pi )^{\ell+1} (n-k)^{\ell+1}} =- \frac{\gamma \, P^m  \, m!}{(2\pi {\mathrm i})^{m+1} \, (k-n)^{m+1}}\, \left( \sum\limits_{\ell=1}^m \frac{1}{\ell !} \, (2\pi {\mathrm i})^{\ell} \, (k- n)^{\ell}\right).
\end{align*}
\end{proof}

Thus we also obtain
\begin{corollary}\label{cor1}
Let $y(t)= \sum\limits _{j=1}^{M} y_{j}(t)$ be of the form  $(\ref{func-def-5})$, where
$y_j(t)= \left(\sum\limits_{m=0}^{n_j}\gamma_{j,m} \, t^{m}\right) {\mathrm e}^{2\pi \lambda_j t}$ with $n_j\in \nn_0$, $\lambda_j \in \cc$, and  $\gamma_{j,m} \in \cc$ with $\gamma_{j,n_j} \neq 0$. Then the Fourier coefficients of $y_{j}(t)$ with respect to the Fourier series expansion on $[0,P]$ 
are of the following form:\\
1) For $-{\mathrm i} \lambda_j P \not \in \zz$, we have 
$$
c_k(y_j) = \sum_{\ell=0}^{n_j} \frac{A_{j,\ell}}{(k-C_j)^{\ell+1}}, \qquad k \in {\mathbb Z},
$$
where for $j=1, \ldots, M$, and $\ell=0,\ldots,n_j$, 
\begin{equation}\label{CA} C_j := -{\mathrm i} \lambda_j \, P, \  A_{j,\ell} :=  \frac{\ell!}{(2\pi {\mathrm i})^{\ell+1}} \left( \gamma_{j,\ell} \, P^\ell (1-{\mathrm e}^{2 \pi  \lambda_j P} )- {\mathrm e}^{2 \pi  \lambda_j P} \sum_{m=\ell+1}^{n_j}  P^m \binom{m}{\ell} \, \gamma_{j,m}\right) .
\end{equation}
2) For $-{\mathrm i} \lambda_j P = k_j \in \zz$ we have
\begin{equation} \label{prop-per}
c_k(y_j)=
\begin{cases}
 \sum\limits_{\ell=0}^{n_{j}} \frac{\gamma_{j, \ell}}{\ell+1} P^\ell , & k=k_{j}, \\
 \sum\limits_{\ell=0}^{n_j-1} \frac{A^\ast_{j,\ell}}{(k-C_j)^{\ell+1}}, & k \in {\mathbb Z} \setminus \{k_j\},
\end{cases}
\end{equation}
where for $j=1, \ldots, M$, and $\ell=0,\ldots,n_j-1$,
\begin{equation} \label{ajlstar}
C_j := -{\mathrm i} \lambda_j \, P=k_j, \ \ A^\ast_{j,\ell} :=-\frac{ \ell! }{(2\pi {\mathrm i})^{\ell+1}} \sum\limits_{m=\ell+1}^{n_j}  P^m  \binom{m}{\ell} \, \gamma_{j,m}.
\end{equation}
  \end{corollary}

\begin{proof}
Using the formula  (\ref{coef-pr2}) in Theorem \ref{thpr3}, it follows for $-{\mathrm i} \lambda_j P \not \in \zz$  that 
\begin{align*}
c_k(y_j) & = \sum_{m=0}^{n_j} \frac{\gamma_{j,m} P^m \, m!}{(2 \pi {\mathrm i})^{m+1} (k+ {\mathrm i}\lambda_j P)^{m+1}} \left( 1 - {\mathrm e}^{2\pi \lambda_j P} \sum_{\ell=0}^m \frac{1}{\ell!} (2\pi {\mathrm i})^\ell (k +{\mathrm i} \lambda_j P)^{\ell} \right) \\
& = \sum_{\ell=0}^{n_j} \frac{\gamma_{j,\ell} P^\ell \, \ell!}{(2 \pi {\mathrm i})^{\ell+1} (k+ {\mathrm i}\lambda_j P)^{\ell+1}} -{\mathrm e}^{2 \pi  \lambda_j P} \sum_{m=0}^{n_j} \sum_{\ell =0}^m \frac{\gamma_{j,m} P^m m!}{\ell! \, (2 \pi {\mathrm i})^{m+1-\ell} \, (k+{\mathrm i}\lambda_j P)^{m+1-\ell}}\\
& = \sum_{\ell=0}^{n_j} \frac{\ell!}{(2\pi {\mathrm i})^{\ell+1} (k+ {\mathrm i}\lambda_j P)^{\ell+1}}
\left( \gamma_{j,\ell} P^\ell - {\mathrm e}^{2 \pi \lambda_j P} \sum_{m=\ell}^{n_j} \gamma_{j,m} \, P^m \, \binom{m}{\ell} \right).
\end{align*}
For the case $-{\mathrm i} \lambda_j P = k_j \in \zz$ the proof is similar to the one above where we use (\ref{coef-pr2-per}) instead of (\ref{coef-pr2}).
\end{proof}

\subsection{Recovery of Extended Exponential Sums with Non-$P$-periodic Terms} 

We consider first the simpler case where $y(t)$  possesses only non-P-periodic terms.

\begin{lemma}\label{lemma5.3}
Assume that $y(t) = \sum\limits_{j=1}^M \Big( \sum\limits_{m=0}^{n_{j}} \gamma_{j,m} t^{m}\Big) {\mathrm e}^{2\pi \lambda_{j}t}$ in $(\ref{func-def-5})$ possesses only non-P-periodic terms, i.e.,  $-{\mathrm i} \lambda_j P \not \in \zz$ for $j=1,\ldots ,M$. 
Further, let $r_{N}(z)$ be a rational function of the form 
\begin{equation}\label{rat1} 
r_N(z) := \frac{p_{N-1}(z)}{q_N(z)} = \sum_{j=1}^M \sum_{\ell=0}^{n_j} \frac{A_{j,\ell}}{(z-C_j)^{\ell+1}}, 
\end{equation}
where $C_{j}$ and $A_{j,\ell}$, $j=1, \ldots , M$, $\ell=0, \ldots , n_{j}$, are given as in $(\ref{CA})$. Then we have $r_{N}(k) = c_{k}(y)$ for $k \in {\mathbb Z}$.
Moreover, the parameters $\lambda_{j}$ and $\gamma_{j,m}$, $j=1, \ldots , M$, $m=0, \ldots , n_{j}$, of $y(t)$ are given by 
\begin{equation}\label{beta1}
 \lambda_j = \frac{{\mathrm i}C_j}{P},  
\end{equation}
and recursively by 
\begin{equation}\label{bs}
\gamma_{j,m}=\frac{(1-{\mathrm e}^{2\pi {\mathrm i} C_j})^{-1}}{P^{m} } \left(
\frac{(2\pi {\mathrm i} )^{m+1}}{m!} A_{j,m}+  {\mathrm e}^{2\pi {\mathrm i} C_j} \sum\limits_{\ell=m+1}^{n_j}   \binom{\ell}{m} P^{\ell} \gamma_{j,\ell} \right)
\quad m=n_{j}, n_{j}-1
, \ldots ,0.
\end{equation}
\end{lemma}

\begin{proof}
Using Corollary \ref{cor1}, we obtain 
\begin{equation}\label{coef-ext-2} 
c_k(y) = \sum_{j=1}^M c_k(y_j) = \sum_{j=1}^M \sum_{\ell=0}^{n_j} \frac{A_{j,\ell}}{(k-C_j)^{\ell+1}}
\end{equation}
for all $k \in {\mathbb Z}$.
Thus, the Fourier coefficients of $y(t)$ in (\ref{func-def-5}) can be represented by $r_N(z)$ in (\ref{rat1}) such that $r_N(k) = c_k(y)$ for $k \in {\mathbb Z}$.
Note that the polynomials $q_N(z)$ and $p_{N-1}(z)$ determining $r_{N}(z)$ are coprime.  
In particular, $y(t)$ is completely determined by $r_N(z)$.
From (\ref{CA}) we obtain 
 $\lambda_j = \frac{{\mathrm i}C_j}{P}$ for $j=1, \ldots , M$,  
and, taking the vectors ${\mathbf A}_j := (A_{j,0}, \ldots , A_{j,n_j})^T$ and $\gamra_j := (\gamma_{j,0}, \ldots , \gamma_{j,n_j})^T$ for $j=1, \ldots , M$, we conclude from (\ref{CA}) for $C_j = -{\mathrm i}P \lambda_j \not\in {\mathbb Z}$
\begin{equation} \label{gamma2}
\resizebox{\textwidth}{!}{$ {\mathbf A}_j \!= \! \mathrm{diag}\left(\!\! \frac{-{\mathrm e}^{2 \pi {\mathrm i}C_j} \, \ell!}{(2\pi {\mathrm i})^{\ell+1}} \right)_{\ell=0}^{n_j} \left(\!\! \begin{array}{ccccc}
\!(1-{\mathrm e}^{-2 \pi {\mathrm i}C_j}) & \binom{1}{0}P & \binom{2}{0}P^2 & \ldots  & \binom{n_j}{0}P^{n_j}\\[1ex]
0 & \!\!\!(1-{\mathrm e}^{-2 \pi {\mathrm i}C_j})P & \binom{2}{1} P^2 & \ldots & \binom{n_j}{1}P^{n_j} \\
\vdots &   & &  & \vdots \\
\vdots &  & & \ddots & \binom{n_j}{n_j-1} P^{n_j}\\[1ex]
0 & 0 &\ldots & 0 & \!\!\!(1-{\mathrm e}^{-2 \pi {\mathrm i}C_j}) P^{n_j} \end{array} \!\!\!\right) 
\gamra_j. 
$}
\end{equation}
The assumption $C_j = -{\mathrm i} P \lambda_j \not\in {\mathbb Z}$ yields invertibility of the upper triangular matrix in the formula above, and moreover, we can compute $\gamra_j$ recursively by  backward substitution to get (\ref{bs}). 
\end{proof}

Therefore, it suffices to determine the parameters of $r_N(z)$ to reconstruct $y(t)$. We obtain the following generalization of Theorem \ref{theo2}.

\begin{theorem}\label{theo4}  Let $y$ be of the form $(\ref{func-def-5})$ with $N = \sum_{j=1}^M (1+n_j) \in {\mathbb N}$, pairwise distinct $ \lambda_{j} \in {\mathbb C} \setminus \frac{{\mathrm i}}{P} {\mathbb Z}$ and  $\gamma_{j,m} \in {\mathbb C}$ for $j=1, \ldots , M$, $m=0, \ldots , n_j$, with  $\gamma_{j,n_j} \neq 0$.
Let $\{ c_{k}(y): \, k \in \Gamma\}$ with $\Gamma \subset {\mathbb Z}$ be a set of $L \ge 2N+1$ Fourier coefficients of the Fourier expansion of $y$ on the finite interval $[0,P] \subset {\mathbb R}$ with $P>0$.
Then $y$ is uniquely determined by $2N$ of these Fourier
coefficients and Algorithm $\ref{alg1}$ $($with ${\mathbf \Gamma} = (k)_{k \in \Gamma}$ and ${\mathbf f} := (c_{k}(y))_{k \in \Gamma})$  terminates after $N$ steps taking $N + 1$ interpolation points and provides a rational function $r_{N}(z)$ that satisfies  ${c}_{k}(y)= r_{N}(k)$ for all $k \in {\mathbb Z}$.
\end{theorem}

\begin{proof} 
As shown in Lemma \ref{lemma5.3}, the Fourier coefficients of $y(t)$ in (\ref{func-def-5}) admit a representation of the form (\ref{coef-ext-2}), i.e., there exists a rational function $r_N(z)$ of type $(N-1,N)$ such that we have $r_N(k)=c_k(y)$. This rational function is already determined by $2N$ interpolation conditions. 
To show that Algorithm \ref{alg1} provides this rational function (in barycentric form) after $N$ iteration steps, we again have to inspect the generalized Cauchy matrix ${\mathbf A}_{N+1} = \left( \frac{c_n(y)- c_k(y)}{n-k} \right)_{n \in \Gamma \setminus S, k \in S}$, where $S \subset \Gamma$ denotes the index set for interpolation. Using similar arguments as in the proof of Theorem \ref{theo2}, we can show that ${\mathbf A}_{N+1}$ possesses exactly rank $N$ and that the vector ${\mathbf w} \in {\mathbb C}^{N+1}$ satisfying ${\mathbf A}_{N+1}\, {\mathbf w} = {\mathbf 0}$ is the weight vector determining the rational function $r_N(z)$ in the barycentric form (\ref{bar-form-1}).
\end{proof}

To reconstruct $y(t)$ in (\ref{func-def-5}) we now can proceed as follows.

\noindent
\textbf{Step 1.} First, 
we apply Algorithm \ref{alg1} with  $P>0$, the period for the computation of Fourier coefficients,  $\mathbf {\Gamma} \in \zz^{L}$, the vector of indices of given Fourier coefficients, and
$\mathbf{f}=\mathbf{c}  = (c_k(y))_{k \in \Gamma} \in \cc^{L}$, the vector of given  Fourier coefficients.
After $N$ iteration steps, we obtain $r_N(z)$ in barycentric form (\ref{bar-form-1}) that is determined by the vector of interpolation indices ${\mathbf S} =(k_1,\ldots,k_{N+1})^T$, the vector of corresponding Fourier coefficients ${\mathbf f}_{\mathbf S}= (c_{k_\ell}(y))_{\ell=1}^{N+1}$,  and the weight vector $\mathbf{w}=\left(w_j \right)_{j=1}^{N+1}$.

\noindent
\textbf{Step 2.} The rational function $r_N(z)$ can be rewritten as a partial fraction decomposition 
$$
r_N(z)=\sum\limits_{j=1}^M \sum\limits_{\ell=0}^{n_j}    \frac{A_{j,\ell}}{(z-C_j)^{\ell+1}}.
$$
To recover the parameters $C_j$ and $A_{j,\ell}$, we need to apply a modified version of Algorithm \ref{alg2}. As before,
  the parameters $C_j$ are the poles of $r_{N}(z)$, i.e., the zeros of $\tilde{q}_{N}(z)$. 
This time, the zeros $C_j$ may come with multiplicity $n_j+1 \ge 1$. 
Afterwards, the parameters $A_{j,\ell}$ are computed using the interpolation conditions $c_{k_\ell}(y) = r_N(k_\ell)$, $\ell=1, \ldots, N+1$. 

We summarize the modified Algorithm \ref{alg4}.
\begin{algorithm} [Reconstruction  of a partial fraction representation]
\label{alg4}
\textbf{Input: } $\mathbf{S} \in \zz^{N+1}$,
  $\mathbf{c}_{\mathbf{S}} \in \cc^{N+1}$,
   $\mathbf{w} \in \cc^{N+1}$ the output vectors of Algorithm \ref{alg1}.

\begin{itemize}
\item Build the matrices in (\ref{eig}) with ${\mathbf S} = (z_j)_{j=1}^{N+1} = (k_j)_{j=1}^{N+1}$ and ${\mathbf w} = (w_{j})_{j=1}^{N+1}$ and solve this eigenvalue problem to find the parameter vector $\mathbf{C} :=\boldsymbol{ \rho}=(C_1, \ldots,C_N)^{T}$ of finite eigenvalues. Extract the number $M$ of different poles $\rho_{j}=C_j$ and the  corresponding multiplicities $n_1,\ldots,n_M$.
\item Solve the linear system with ${\mathbf c}_{\mathbf S} = (c_{k_{\nu}})_{\nu=1}^{N+1}$, 
$$\sum\limits_{j=1}^M \sum\limits_{\ell=0}^{n_j}    \frac{A_{j,\ell}}{(k_\nu-C_j)^{\ell+1}}=c_{k_\nu}(y), \qquad  \nu=1,\ldots,N+1. $$
\end{itemize}
\textbf{Output: } $M$ the number of different poles $C_j$, \\
\phantom{Output:} parameter vectors  $(C_j)_{j=1}^{M}$, $\left(n_j \right)_{j=1}^{M}$, and $(A_{j,\ell})_{j=1,\ldots,M, \, \ell=0,\ldots,n_j}$.

\end{algorithm}

\noindent
\textbf{Step 3.} Finally, we extract the wanted parameters $\lambda_j$, $\gamma_{j,m}$ via (\ref{beta1}) and (\ref{gamma2}).

\begin{example}
We consider the extended exponential sum 
\begin{align}
y_3(t):= &  \left(3.1+0.5 \, \mathrm{i} +0.5t-0.002t^{2}+1.6 t^{3}+(0.55-4.23  \, \mathrm{i}) t^{4} \right) {\mathrm e}^{2\pi  \,( -0.1236+2.2371 \mathrm{i} )t} \notag \\
&-15.02{\mathrm e}^{2\pi  \, ( 0.011-\sqrt{2.2} \mathrm{i})t},  \label{y3}
\end{align}
see Figure \ref{fig3}, i.e, $y_3$ is of the form  (\ref{func-def-5}) with $N=6$, and with parameters
$$
{\mathbf n}=(n_1, \, n_2) = (4,0), \qquad  \lamdra=(\lambda_1, \lambda_2) = (-0.1236+2.2371 \mathrm{i} , 0.011-\sqrt{2.2} \mathrm{i}),
$$ 
and $\gamra=(\gamra_{1}^{T}, \gamra_{2}^{T})$ with 
$$
\gamra_1=(\gamma_{1,0}, \ldots , \gamma_{1,4}) = (3.1+0.5 \, \mathrm{i} , 0.5,-0.002,  1.6, 0.55-4.23  \, \mathrm{i}), \quad   \gamra_2= \gamma_{2,0}= -15.02.
$$

\begin{figure}[h]
\centering
\begin{subfigure}{0.46\textwidth}
\includegraphics[width=\textwidth]{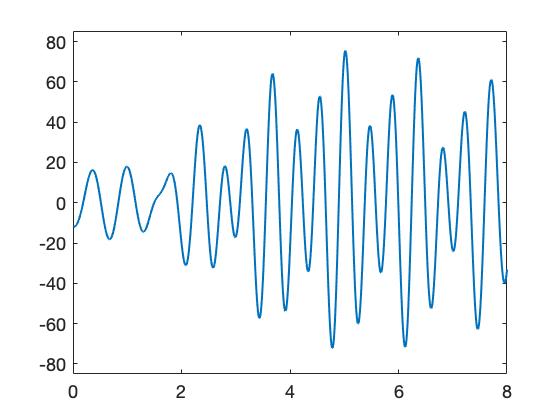}   
\end{subfigure}
\begin{subfigure}{0.46\textwidth}
\includegraphics[width=\textwidth]{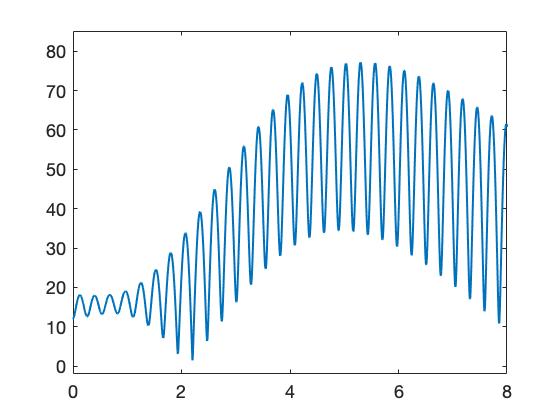}   
\end{subfigure}
\caption{Graph of  $\mathrm{Re} \, y_3$ (left) and of $|y_3|$ (right) for $y_3(t)$ in (\ref{y3}) on 
$[0,8]$.}
\label{fig3}
\end{figure}

We take $P=8$ and use 59 Fourier coefficients $c_k(y_3)$, $k=-29,\ldots,29$, for the recovery of $y_3$. Algorithm \ref{alg1} iteratively uses the values  $c_{18}(y_3)$, $c_{-12}(y_3)$, $c_{17}(y_3)$, $c_{-8}(y_3)$, $c_{19}(y_3)$, $c_{15}(y_3)$ and $c_{21}(y_3)$ for interpolation (in this order) before it terminates with the error ${1.01 \cdot 10^{-14}}$ after 6 iteration steps. We get a rational function $r_6(z)$ of type $(5,6)$ in barycentric form (\ref{bar-form-1})  determined by
$$
 {\mathbf  S}=\left(
\begin{matrix}
    18\\
   -12\\
    17\\
    -8\\
    19\\
    15\\
    21
\end{matrix}
\right), \, \, \, \, \, \, \, \, \, 
 {\mathbf  w}=\left(
\begin{matrix}
  0.002282995536669 + 0.003576768606091\mathrm{i}\\
 -0.002637047964009 + 0.040954493823534\mathrm{i}\\
  0.010814031928894 - 0.000636557543413\mathrm{i}\\
 -0.633666296891693 + 0.762517726922522\mathrm{i}\\
  0.000499981716439 + 0.016884004429586\mathrm{i}\\
  0.070370845980150 + 0.043823205664822\mathrm{i}\\
  0.078294211622157 + 0.043889969381308\mathrm{i}
\end{matrix}
\right),
$$
and the corresponding vector of Fourier coefficients ${\mathbf c}_{\mathbf S}$. We reconstruct the parameter vectors $\tilde{\mathbf n}$, $\tilde{\lamdra}$ and $\tilde{\gamra}$ using Algorithms \ref{alg4} and via (\ref{beta1}) and (\ref{gamma2}).  Algorithm \ref{alg4} yields the poles $C_j=- 8{\mathrm i}\lambda_j$, $j=1,\ldots,7$, of the rational function $r_6(z)$,
$$
 {\mathbf  C}=\left(
\begin{matrix}
-11.865917579353065 - 0.088000000000003\mathrm{i}\\
 17.895308407373371 + 0.984361807596397\mathrm{i}\\
 17.892106783572142 + 0.988839706980242\mathrm{i}\\
 17.900555173098393 + 0.985997347696078\mathrm{i}\\
 17.895392315518535 + 0.993288509746222\mathrm{i}\\
 17.900637320434701 + 0.991512627981242\mathrm{i}
\end{matrix}
\right).
$$

We assume that two computed poles $C_{j_1}$ and $C_{j_2}$, $j_1\neq j_2$, are equal, if  $|C_{j_1}-C_{j_2}| < 0.01$. 
 We  obtain the pole $C_1=  17.896799999999427 + 0.988800000000036\mathrm{i}$  with  multiplicity 5 and the pole $C_2=-11.865917579353065 - 0.088000000000003\mathrm{i}$ with multiplicity 1. Note that $C_1$ is taken as the average of the last $5$ values in ${\mathbf C}$, i.e., ${\mathbf n} = (4,0)$.
The reconstruction procedure provides the errors
$$
\|\tilde{\lamdra} - \lamdra\|_\infty=7.16 \cdot  10^{-14}, \qquad  \|\tilde{\gamra} - \gamra\|_\infty=2.34 \cdot 10^{-10},
$$
where $\tilde{\lamdra}$ and $\tilde{\gamra}$ denote  the computed parameter vectors. 
\end{example}

\subsection{Recovery of Extended Exponential Sums Containing also $P$-periodic Terms} 
\label{sec:52}

Now let us assume that the function (\ref{func-def-5}) contains also $P$-periodic components. 
In this case it can be represented as $y=y^{(1)}+y^{(2)}$, where
\begin{equation}\label{func-def-y1}
y^{(1)}(t):=\sum_{j=1}^{M_1} \left( \sum_{m=0}^{n_j}  \, \gamma_{j,m} \, t^{m} \right) e^{2 \pi  \lambda_jt}, \ \  -{\mathrm i} \lambda_j P \not \in \zz,
\end{equation}
and
\begin{equation}\label{func-def-y2}
y^{(2)}(t):=\sum_{j=M_1+1}^{M} \left( \sum_{m=0}^{n_j}  \, \gamma_{j,m} \, t^{m} \right) e^{2 \pi  \lambda_jt}, \ \  -{\mathrm i} \lambda_j P  \in \zz.
\end{equation}
Let $M_{1}$ and $N_1:=\sum\limits_{j=1}^{M_1}(1+n_j)$ be the length and the order of $y^{(1)}$ and let $M_2:=M-M_1$ and $N_2:=\sum\limits_{j=M_1+1}^{M}(1+n_j)=N-N_1$ denote the length and the order of $y^{(2)}(t)$. We define the set $\Sigma:=\{-{\mathrm i} \lambda_j P: \, j=M_{1}+1, \ldots , M\}$ corresponding to the $M_{2}$ periodic terms in $y^{(2)}$.

\begin{lemma}\label{lemma5.4}
Let $y= y^{(1)}+ y^{(2)}$ with $y^{(1)}$ in  $(\ref{func-def-y1})$ and $y^{(2)}$ in  $(\ref{func-def-y2})$.
Define
\begin{equation}\label{rat-per} 
r^{\dagger}_{N-M_{2}}(z) := r_{N_{1}}(z) + r^{\ast}_{N_2-M_2}(z) = \sum_{j=1}^{M_{1}} \sum_{\ell=0}^{n_j} \frac{A_{j,\ell}}{(z-C_j)^{\ell+1}} + \sum_{j=M_{1}+1}^M \sum_{\ell=0}^{n_j-1} \frac{A^{\ast}_{j,\ell}}{(z-C_j)^{\ell+1}}, 
\end{equation}
where $C_{j}$ and $A_{j,\ell}$, $j=1, \ldots , M_{1}$, $\ell=0, \ldots , n_{j}$,  are given as in $(\ref{CA})$, and $C_{j}$ and $A^{\ast}_{j,\ell}$, $j={M_{1}+1}, \ldots , M$, $\ell=0, \ldots , n_{j}-1$, as in $(\ref{ajlstar})$ (where the range for $j$ has to be adjusted).
Then we have 
$$ c_k(y)=r^{\dagger}_{N-M_{2}}(k) \quad \textrm{ for}  \qquad k\in {\mathbb Z}\setminus \Sigma, $$ 
and in particular, $c_k(y^{(1)}) = r_{N_{1}}(k)$, $k\in {\mathbb Z}$.
 Moreover, the parameters $\lambda_{j}$, $j=1, \ldots , M$, are determined by
$(\ref{beta1})$. Further, $\gamma_{j,m}$, $j=1, \ldots , M_{1}$, $m=0, \ldots , n_{j}$, of $y^{(1)}$ 
are given as in  $(\ref{bs})$, and  $\gamma_{j,m}$, $j=M_{1}+1, \ldots , M$, $m=1, \ldots , n_{j}$, of $y^{(2)}$ are 
 recursively given by 
\begin{equation}\label{bs1}
\gamma_{j,m+1}= - \frac{1}{P^{m+1} (m+1)} \left( \frac{(2\pi {\mathrm i} )^{m+1} A_{j,m}^{\ast}}{m!} +  \sum\limits_{\ell=m+2}^{n_j}   \binom{\ell}{m} P^{\ell} \, \gamma_{j,\ell} \right), \quad m=n_{j}-1, 
, \ldots ,0.
\end{equation}
The parameters $\gamma_{j,0}$, $j=M_{1}+1, \ldots , M$, are determined with $k_{j}:= - {\mathrm i} \lambda_{j} P$ by
$$
\gamma_{j,0} = c_{k_j}(y^{(2)})-\sum\limits_{\substack{ j_1=M_1+1 \\ j_1 \neq j}}^{M}  \sum\limits_{\ell=0}^{n_{j_1}-1} \frac{A^\ast_{j_1,\ell}}{(k_{j_1}-C_j)^{\ell+1}} - \sum_{\ell=1}^{n_{j}} \frac{P^{\ell}}{\ell+1} \gamma_{j,\ell}.
$$
\end{lemma}

\begin{proof}
1.\ The Fourier coefficients of the non-$P$-periodic part $y^{(1)}(t)$ can be determined  by the rational function $r_{N_1}(z)$  as in (\ref{rat1}) with $N_1:=\sum\limits_{j=1}^{M_1}(1+n_j)$ such that 
\begin{equation}\label{rat-non-per}
c_k(y^{(1)})=r_{N_1}(k) \ \ \ \text{ for} \  k \in \zz.
\end{equation}
In particular, by Lemma \ref{lemma5.3} there is a bijective map between the parameters $\{n_j, \,  \lambda_j, \, \gamma_{j,m}: \, j=1, \ldots, M_1, \, m=0, \ldots, n_j\}$ determining $y^{(1)}$ in (\ref{func-def-y1}) and  the parameters $\{ n_j, \, C_j, \, A_{j,\ell}: \, j=1, \ldots, M_1, \, \ell=0, \ldots , n_j\}$ determining $r_{N_1}(z)$. Therefore the non-$P$-periodic part $y^{(1)}(t)$ is uniquely determined by the rational function $r_{N_1}(z)$.

2.\ We consider now the representation for the Fourier coefficients of the $P$-periodic part $y^{(2)}(t)$. We denote  $C_j:=-{\mathrm i} \lambda_j P=k_j \in \zz$ for $j=M_1+1,\ldots ,M$, then we have $\Sigma=\{ k_{M_1+1},\ldots ,k_M\}$. According to (\ref{prop-per}),
\begin{equation} \label{prop-per1}
c_k(y^{(2)})=
\begin{cases}
\sum\limits_{\substack{ j=M_1+1 \\ j \neq j_0}}^{M} \sum\limits_{\ell=0}^{n_j-1} \frac{A^\ast_{j,\ell}}{(k-C_j)^{\ell+1}}+ \sum\limits_{\ell=0}^{n_{j_0}} \frac{\gamma_{j_0, \ell}}{\ell+1} P^\ell , & k=k_{j_0} \in \Sigma,\\
\sum\limits_{j =M_1+1}^{M} \sum\limits_{\ell=0}^{n_j-1} \frac{A^\ast_{j,\ell}}{(k-C_j)^{\ell+1}}, & k \in {\mathbb Z} \setminus \Sigma
\end{cases}
\end{equation}
with $A_{j,\ell}^{\ast} $ as in (\ref{ajlstar}).
Thus all Fourier coefficients $c_k(y^{(2)})$, $k \in \zz \setminus \Sigma$, still have a rational structure.  We consider the rational function 
\begin{equation} \label{rat-per}
 r^{\ast}_{N_2-M_2}(z):=\frac{p^{\ast}_{N_2-M_2-1}(z)}{q^{\ast}_{N_2-M_2}(z)} =  \sum_{j=M_{1}+1}^M \sum_{\ell=0}^{n_j-1} \frac{A^{\ast}_{j,\ell}}{(z-C_j)^{\ell+1}}.
\end{equation}
Note that the polynomials $p^{\ast}_{N_2-M_2-1}$ and $q^{\ast}_{N_2-M_2}$ do not have common zeros and $ r^{\ast}_{N_2-M_2}$ is of type $(N_2-M_2-1,N_2-M_2)$. Taking into account (\ref{prop-per1}) and (\ref{rat-per}) we conclude that
\begin{equation}\label{rat-per1}
c_k(y^{(2)})=r^{\ast}_{N_2-M_2}(k), \ \ \ k \in \zz \setminus \Sigma,
\end{equation}
and thus also  $c_{k}(y) = r_{N-M_{2}}^{\dagger}(k)$ for $k \in {\mathbb Z} \setminus \Sigma$.
 
3.\ We show now that the P-periodic function $y^{(2)}(t)$  is uniquely determined  by the rational function 
 $r^{\ast}_{N_2-M_2}$ and the Fourier coefficients $c_{k_{j}}(y^{(2)})$, $k_{j} \in \Sigma$,  or equivalently, by $C_{j}$, $A_{j,\ell}^{\ast}$, and $c_{k_{j}}(y^{(2)})$, $j=M_{1}, \ldots , M$, $\ell=0, \ldots , n_{j}-1$.
The  frequency parameters $\lambda_j$ of $y^{(2)}(t)$ in (\ref{func-def-y2})  are given  by $\lambda_{j} = {\mathrm i} \frac{C_{j}}{P}$ as in (\ref{beta1}), and $k_{j} = C_{j}$.
Observe that the definition of $A_{j,\ell}^{\ast}$, $\ell=0, \ldots, n_{j}-1$, in (\ref{ajlstar}) involves  $\gamma_{j,m}$ for $m=1, \ldots , n_{j}$, and all $\gamma_{j,m}$, $m=1, \ldots , n_{j}$, can be recovered from the $A_{j,\ell}^{\ast}$, $\ell=0, \ldots, n_{j}-1$, recursively. To determine also $\gamma_{j,0}$, $j=M_{1}+1, \ldots , M$,
we have to employ $c_{k_{j}}(y^{(2)})$ and obtain from the first line of (\ref{prop-per1})
$$
\gamma_{j,0} = c_{k_j}(y^{(2)})-\sum\limits_{\substack{ j_1=M_1+1 \\ j_1 \neq j}}^{M}  \sum\limits_{\ell=0}^{n_{j_1}-1} \frac{A^\ast_{j_1,\ell}}{(k_{j_1}-C_j)^{\ell+1}} - \sum_{\ell=1}^{n_{j}} \frac{P^{\ell}}{\ell+1} \gamma_{j,\ell}.
$$
More precisely, to recover all parameters $\gamma_{j,\ell}$ of $y^{(2)}(t)$ let 
\begin{align} \nonumber
\breve{c}_{k_j}(y) &:= c_{k_j}(y^{(2)})-\sum\limits_{\substack{j_1=M_1+1 \\ j_1 \neq j}}^{M}  \sum\limits_{\ell=0}^{n_{j_1}-1} \frac{A^\ast_{j_1,\ell}}{(k_{j_1}-C_j)^{\ell+1}} \\
\label{c-coef}
& = c_{k_j}(y)-c_{k_j}(y^{(1)}) -\sum\limits_{\substack{ j_1=M_1+1 \\ j_1 \neq j}}^{M}  \sum\limits_{\ell=0}^{n_{j_1}-1} \frac{A^\ast_{j_1,\ell}}{(k_{j_1}-C_j)^{\ell+1}}.
\end{align}
We define the
vectors  ${\mathbf A}_j^{\ast} := (\breve{c}_{k_j}(y),A^{\ast}_{j,0}, \ldots , A^{\ast}_{j,n_j-1})^T$ and  $\gamra_{j} :=(\gamma_{j,0}, \ldots , \gamma_{j,n_{j}})^{T}$. 
Then we obtain the linear relation 
\begin{equation}\label{sys-eq-per}
{\mathbf A}_j^{\ast} =
\left( \begin{array}{ccccc}
1 & \frac{P}{2} & \frac{P^{2}}{3} & \ldots & \frac{P^{n_j}}{n_j+1} \\
0 & -\binom{1}{0}P &  -\binom{2}{0}P^{2} &  \ldots & -\binom{n_j}{0}P^{n_j} \\
0 & 0 & -\frac{1}{2\pi {\mathrm i} } \binom{2}{1}P^{2}&  \ldots & -\frac{1}{2\pi {\mathrm i} } \binom{n_j}{1}P^{n_j}\\
\vdots & &  \ddots & & \\
0 & 0 &  \ldots &  0 &  -\frac{1}{(2\pi {\mathrm i})^{n_j} } n_j ! P^{n_j} 
\end{array} \right)  \gamra_j. 
\end{equation}  
Note also that the matrix in (\ref{sys-eq-per}) is invertible, therefore this system  has a unique solution $\gamra_j$.
For the case when some $n_j=1$, i.e., when the extended exponential sum (\ref{func-def-5}) has a proper component, we get that $\gamra_{j} =\gamma_{j,0}$ and (\ref{sys-eq-per}) simply gives (\ref{coef-rec-prop}).
\end{proof}

Finally, we study the recovery of extended exponential sums $y$ in (\ref{func-def-5}) that can be written as 
 $y=y^{(1)}+y^{(2)}$, where $y^{(1)}$ is the non-$P$-periodic part defined by (\ref{func-def-y1}) and $y^{(2)}$ is the $P$-periodic part defined by (\ref{func-def-y2}).
If $y^{(2)}$ is not just a proper exponential sum then this recovery problem essentially differs from the recovery of proper exponential sums in Section \ref{sec:periodic},
since $y^{(2)}$ also possesses infinitely many nonzero Fourier coefficients if we have some $n_j >0$ for $j \in \{M_1+1, \ldots , M\}$, see 
(\ref{prop-per1}). 
But by Lemma \ref{lemma5.4}, all but $M_2$ Fourier coefficients still have the structure of a rational function. 
Using this information, we can now reconstruct the function $y= y^{(1)} + y^{(2)}$ as follows.
\begin{theorem} \label{theoper1}
Let $y$ in $(\ref{func-def-5})$ be of the form $y = y^{(1)} + y^{(2)}$ as in $(\ref{func-def-y1})$ and $(\ref{func-def-y2})$, with $\gamma_{j,m} \in {\mathbb C}$, $\gamma_{j,n_{j}} \neq 0$, and where  $ \lambda_{j} \in {\mathbb C}$ are pairwise distinct. Further, let  $M_1 < M$, and let $M_2:=M-M_1$ be the number of $P$-periodic components in $y^{(2)}$.
Denote by $\Sigma:= \{-{\mathrm i} \lambda_{j} P: \, j=M_{1}+1, \ldots , M\} \subset {\mathbb Z}$ the index set corresponding to the frequencies of the $P$-periodic  part $y^{(2)}$.
Let $\{c_{k}(y): \, k \in \Gamma\}$, with $\Gamma \subset {\mathbb Z}$ be a set of $L \ge 2N+2$  Fourier coefficients of the Fourier expansion of $y$ on the finite interval $[0,P] \subset {\mathbb R}$ with $P>0$. Assume that  $\Sigma \subset \Gamma$. 
Then $y$ can be uniquely recovered from this set of Fourier coefficients. Algorithm $\ref{alg1}$ $($with ${\mathbf \Gamma} = (k)_{k \in \Gamma}$ and ${\mathbf f} := ({c}_{k}(y))_{k \in \Gamma}) $ terminates after at most $N+1$ steps  and provides a rational function $r^{\dagger}_{N-M_2}(z)$ of type $(N-M_2-1,N-M_2)$ that satisfies ${c}_{k}(y)= r^{\dagger}_{N-M_2}(k)$ for all $k \in {\mathbb Z} \setminus \Sigma$.
\end{theorem}

Theorem \ref{theoper1} can be proved along the lines of Theorem \ref{theoper}.
To recover $y$, we can now proceed as follows.

\noindent
\textbf{Step 1.} We apply Algorithm \ref{alg1} to ${\mathbf \Gamma} = (k)_{k \in \Gamma}$ and ${\mathbf f} := ({c}_{k}(y))_{k \in \Gamma}$ and find after at most  $N+1$ iteration steps the (sub)vectors ${\mathbf S}= (z_j)_{j=1}^{N+1}= (k_{\ell})_{\ell=1}^{N+1}$, ${\mathbf c}_{\mathbf S} = (c_{k_{\ell}}(y))_{\ell=0}^{N+1} \in {\mathbb C}^{N+1}$ and ${\mathbf w}= (w_{j})_{j=1}^{N+1} \in {\mathbb C}^{N+1}$. If ${\mathbf S}$ contains integer indices $k_j \in \zz$ of the form $k_j=-{\mathrm i}\lambda_j P$, $j \in \{M_1+1, \ldots , M\}$, then the corresponding components of ${\mathbf w}$ vanish, since these interpolation points do not possess a rational function structure. 
Therefore, we simply remove the zero components of ${\mathbf w}$ and the corresponding components of ${\mathbf S}$ and ${\mathbf c}_{\mathbf S}$ to obtain the parameter vectors that determine the rational function $r^{\dagger}_{N-M_2}(z)$ of type $(N-M_2-1,N-M_2)$ in barycentric form (\ref{bar-form-1}).

\noindent
\textbf{Step 2.} 
We apply now a  modification of Algorithms \ref{alg4} as follows.  From the definition of $r^{\dagger}_{N-M_2}(z)$ in Lemma \ref{lemma5.4} it follows that we find $N-M_2$ poles (counting also multiplicities):  each $C_j \not\in {\mathbb Z}$  
is a pole with multiplicity $n_j+1$ and belongs to rational function  $r_{N_1}(z)$ that corresponds to the non-$P$-periodic part $y^{(1)}$. Each $C_j \in {\mathbb Z}$ is a pole with multiplicity $n_j$ and belongs to the rational function $r^{\ast}_{N_2-M_2}(z)$ which corresponds to
the $P$-periodic part $y^{(2)}$.  Taking into account this information we find
 $M$ pairwise distinct $C_j$ that are the poles of the $r^{\dagger}_{N-M_2}(z)$ with multiplicities $n_j+1$ for non-$P$-periodic components and $n_j$ for $P$-periodic components. The set $\Sigma$ can be easily determined by  $\Sigma=\{C_j: \, j=1,\ldots,M\} \cap \zz$.  Then we compute the parameters $A_{j,\ell}$, $j=1,\ldots,M_1$, $\ell=0,\ldots,n_j$, and $A^\ast_{j,\ell}$, $j=M_1+1,\ldots,M$, $\ell=0,\ldots,n_j-1$, using the $N+1-M_2$ interpolation conditions 
$$
r^{\dagger}_{N-M_2}(k_s)=c_{k_s}(y), \ \ s=1,\ldots, N+1-M_2, \ \ k_s \not \in \Sigma.
$$
The values $A_{j,\ell}$, $j=1,\ldots,M_1$, $\ell=0,\ldots,n_j$, and $A^\ast_{j,\ell}$, $j=M_1+1,\ldots,M$, $\ell=0,\ldots,n_j-1$, are solutions of this system. In general we have $N-M_2$ values $A_{j,\ell}, \, A^\ast_{j,\ell}$.

\noindent
\textbf{Step 3.} We recover the frequencies $\lambda_j$ via (\ref{beta1})  for all $j=1,\ldots ,M$. For the non-$P$-periodic function $y^{(1)}$,  we find coefficients $\gamma_{j,m}$, $j=1,\ldots ,M_1$, $m=0,\ldots ,n_j$, by using formula (\ref{gamma2}) with values $A_{j,\ell}$, $j=1,\ldots ,M_1$, $\ell=0,\ldots ,n_j$.

\noindent
\textbf{Step 4.} 
We  compute the coefficients $\gamma_{j,m}$, $j=M_1+1,\ldots ,M$, $m=0,\ldots ,n_j$, by solving the linear system (\ref{sys-eq-per}) by using values $A^\ast_{j,\ell}$, $j=M_1+1,\ldots ,M$, $\ell=0,\ldots ,n_j-1$, and the vector  $\breve{\mathbf{c}}_{\Sigma}= (\breve{c}_{k}(y))_{k \in \Sigma}$, with $\breve{c}_{k}(y)$ in (\ref{c-coef}),  which can be determined using $c_{k}(y^{(1)}) = \sum\limits_{j=1}^{M_1}  \sum\limits_{\ell=0}^{n_j}   \frac{A_{j,\ell}}{(k-C_j)^{\ell+1}}$ for $k \in \Sigma$. 

When the exponential sum contains a proper $P$-periodic part, the corresponding component $-\mathrm{i} \lambda_j P$ does not appear as a pole because in this case $n_j=0$. Therefore, we apply a technique  similar to the one that we used in Section \ref{sec:periodic}  in order to detect the corresponding frequency  $\lambda_j$, namely we compere the Fourier coefficients and the values of the rational function constructed by Algorithm \ref{alg1}. The coefficient $\gamma_j$ can be found then via  (\ref{coef-rec-prop}).

\begin{example}\label{mulrper}
We consider the extended exponential sum 
\begin{align}
y_4(t):=&\left( 3.46-0.5 \mathrm{i}+ (-1.6+7.3 \mathrm{i})t-2.4 t^2 \right) {\mathrm e}^{2\pi\cdot ( -0.1-0.73\mathrm{i})  t} \nonumber \\
&+ \left( -3.8-1.999\mathrm{i}+(-0.2-0.4\mathrm{i})t \right) {\mathrm e}^{2\pi \cdot(0.05- \sqrt{10.11} \mathrm{i}) t} \nonumber \\
& + \left( -7.33+7.033 \mathrm{i} + 3.89 t + (2.48-0.45 \mathrm{i})t^2 + (-5.3+0.01 \mathrm{i})t^3 \right) {\mathrm e}^{2\pi \cdot 1.5 \mathrm{i}  t}, \label{y4}
\end{align}
see Figure \ref{fig4},
i.e., $y_4$ is of the form (\ref{func-def-5}) with $N=9$ and with parameters
\begin{align*}
{\mathbf n} &= (n_1,n_2, n_3) =(2,1,3), \qquad \lamdra=(\lambda_1, \lambda_2, \lambda_3) = (-0.1-0.73\mathrm{i}, 0.05- \sqrt{10.11} \mathrm{i}, 1.5 {\mathrm i}), \\
\gamra_1  &= (\gamma_{1,0}, \gamma_{1,1}, \gamma_{1,2}) = (3.46-0.5 \mathrm{i}, -1.6+7.3 \mathrm{i}, -2.4 ), \\
\gamra_2 &=(\gamma_{2,0}, \gamma_{2,1}) = ( -3.8-1.999\mathrm{i}, -0.2-0.4\mathrm{i}), \\
 \gamra_3 &=(\gamma_{3,0}, \gamma_{3,1}, \gamma_{3,2}, \gamma_{3,3}) = (-7.33+7.033 \mathrm{i}, 3.89,  2.48-0.45 \mathrm{i}, -5.3+0.01 \mathrm{i}).
\end{align*}
\begin{figure}[h]
\centering
\begin{subfigure}{0.46\textwidth}
\includegraphics[width=\textwidth]{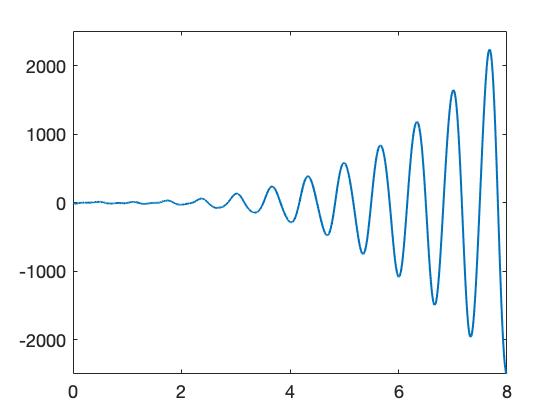}   
\end{subfigure}
\begin{subfigure}{0.46\textwidth}
\includegraphics[width=\textwidth]{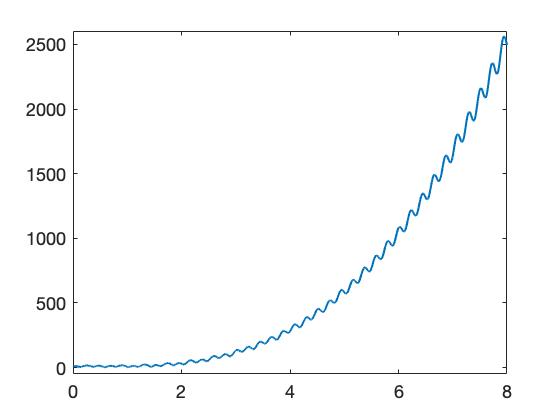}   
\end{subfigure}
\caption{Graph of  $\mathrm{Re} \, y_4$ (left) and of $|y_4|$ (right) for $y_4(t)$ in (\ref{y4}) on 
$[0,8]$. }
\label{fig4}
\end{figure}
To reconstruct $y_{4}$, we  take $P=8$ and employ 95 Fourier coefficients $c_k(y_4)$, $k=-47,\ldots,47$.  For $P=8$, the function $y_{4}$ contains two extended non-$P$-periodic terms,
\begin{align*}
y^{(1)}_4(t)=&\left( 3.46-0.5 \mathrm{i}+ (-1.6+7.3 \mathrm{i})t-2.4 t^2 \right) {\mathrm e}^{2\pi\cdot (-0.1-0.73\mathrm{i})t} \\
& + \left( -3.8-1.999\mathrm{i}+(-0.2-0.4\mathrm{i})t \right) {\mathrm e}^{2\pi  \cdot(0.05- \sqrt{10.11} \mathrm{i})t},
\end{align*}
and one extended $P$-periodic term with $n_3=3$,
$$
y^{(2)}_4(t) = \left( -7.33+7.033 \mathrm{i} + 3.89 t + (2.48-0.45\mathrm{i})t^2 + (-5.3+0.01 \mathrm{i})t^3 \right) {\mathrm e}^{2\pi \cdot 1.5 \mathrm{i} t}.
$$
Algorithm \ref{alg1} iteratively uses the values  $c_{12}(y_4)$, $c_{11}(y_4)$, $c_{13}(y_4)$,  $c_{-25}(y_4)$, $c_{-26}(y_4)$, $c_{-6}(y_4)$, $c_{-5}(y_4)$, $c_{-7}(y_4)$, $c_{15}(y_4)$ and $c_{27}(y_4)$ for interpolation (in this order) before it stops with the error $3.35\cdot 10^{-13}$ after $9$ iteration steps. 
The first component of ${\mathbf  w} \in \cc^{10}$ vanishes showing that $c_{12}(y_4)$  is not interpolated by the obtained rational function. Indeed for the frequency $\lambda_3= 1.5{\mathrm i} $ we have $-{\mathrm i} \lambda_3 P=12$. Therefore the coefficient $c_{12}(y_4)$ contains information about the $P$-periodic part $y^{(2)}_4$ of $y_4$. After omitting this first term in  ${\mathbf  w} $ and in the corresponding index in the vector ${\mathbf  S}$, we obtain the rational function $r_9(z)$ of type $(8,9)$ in the barycentric form (\ref{bar-form-1}) with
$$
 {\mathbf  S}=\left(
\begin{matrix}
    11\\
    13\\
   -25\\
   -26\\
    -6\\
    -5\\
    -7\\
    15\\
    27
\end{matrix}
\right), \, \, \, \, \, \, \, \, \, 
 {\mathbf  w}=\left(
\begin{matrix}
 -0.001735068914586 - 0.005755006114429{\mathrm i} \\
 -0.003788627020742 - 0.013231843260615{\mathrm i} \\
 -0.005033054805759 + 0.000686121267669{\mathrm i} \\
 -0.006507095079179 - 0.001885228349247{\mathrm i} \\
  0.008318151814712 - 0.007345769247111{\mathrm i} \\
  0.010826619319887 - 0.012212232030143{\mathrm i} \\
  0.027770232215522 + 0.003108429826714{\mathrm i} \\
  0.057427797470378 + 0.209382506645587{\mathrm i} \\
 -0.222536073428981 - 0.949668937525004{\mathrm i} 
\end{matrix}
\right).
$$
We compute the vector ${\mathbf C}$ of poles, the number $M$ of pairwise distinct poles, the vector ${\mathbf n}$ of multiplicities of poles and the parameters $A_{j,\ell}$ for $\ell=0,\ldots ,n_j$, $j=1,\ldots ,M_1$,  $A^{\ast}_{j,\ell}$ for  $\ell=0,\ldots ,n_j-1$,  $j=M_1+1, \ldots , M$, as it is explained in Steps 2 and 3 above.
We obtain the poles of $r_9(z)$
$$
 {\mathbf  C}=\left(
\begin{matrix}
-25.436980930912810 - 0.400000242723915{\mathrm i} \\
-25.436980974957748 - 0.399999757275449{\mathrm i} \\
 12.000005193419646 - 0.000047024733537{\mathrm i} \\
 11.999956677814936 + 0.000019014605692{\mathrm i} \\
 12.000038128765494 + 0.000028010127754{\mathrm i} \\
 -5.840164248641018 + 0.799937410836133{\mathrm i} \\
 -5.839863721435004 + 0.799889033201780{\mathrm i} \\
 -5.839972029901213 + 0.800173555965108{\mathrm i} 
\end{matrix}
\right).
$$

We assume that two poles $C_{j_1}$ and $C_{j_2}$, $j_1\neq j_2$ are equal if $|C_{j_1}-C_{j_2}| < 0.001$ holds, and find  $M=3$ different poles: $C_1=  -5.839999999992412 + 0.800000000001007 {\mathrm i}$ with multiplicity $n_1=3$, $C_2=-25.436980952935279 - 0.399999999999682 {\mathrm i}$ with multiplicity $n_2=2$, and $C_3=  12.000000000000027 - 0.000000000000030i {\mathrm i}$ with multiplicity $n_3=4$. The poles $C_1$ and $C_2$ corresponds to the non-8-periodic part of the exponential sum $y_4$, therefore they appear in $ {\mathbf  C}$  with multiplicities $n_1+1$ and $n_2+1$ respectively. The pole $C_3$ corresponds to the 8-periodic part of $y_4$ therefore it comes with multiplicity $n_3$.
To determine the values for the poles $C_1$, $C_2$ and $C_3$ we have taken the average values. 
Finally, we reconstruct $\lambda_j$, $j=1, 2, 3$, and $\gamma_{j,\ell}$, $j=1,2$, $\ell=0, \ldots , n_{j}$, as described in Step 3 and $\gamma_{3,\ell}$, $\ell=0, \ldots , n_{3}$, via Step 4 above.
The reconstructed parameter vectors $\tilde{\mathbf n}$ and $\tilde{\lamdra}$ read, $\tilde{\mathbf n}=(2,1,3)$
and
$$
 \tilde{\lamdra}^{T}=\left(
\begin{matrix}
  -0.100000000000126 - 0.729999999999052{\mathrm i} \\
  0.049999999999960 - 3.179622619116910{\mathrm i} \\
0.000000000000004 + 1.500000000000003{\mathrm i} 
\end{matrix}
\right).
$$
The recovery errors are
$$
\|\tilde{\lamdra} -\lamdra\|_\infty=9.56 \cdot  10^{-13}, \ \ \ \ \|\tilde{\gamra} - \gamra\|_\infty=3.11 \cdot 10^{-10}.
$$
\end{example}

\section{Conclusions}

In Sections 4 and 5 we have considered the recovery of proper and extended exponential sums. We have shown that sums of the form 
$$
y(t)=\sum_{j=1}^M \left( \sum_{m=0}^{n_j}  \, \gamma_{j,m} \, t^{m} \right) {\mathrm e}^{\lambda_j t}, \quad    \gamma_{j, n_j}\neq 0, 
$$
with $\sum\limits_{j=1}^{M}(1+n_j ) = N$, $\gamma_{j,m}\in \cc$ and pairwise distinct $\lambda_j \in {\mathbb C}$ can always be recovered from at most $2N+2$ Fourier coefficients of a Fourier expansion on a finite interval. Numerical stability of this procedure essentially depends on the  numerical stability of the underlying AAA algorithm for rational approximation of these Fourier coefficients. Observe that the considered model also covers sums of the form
$$
y_0(t)=\sum\limits_{m=0}^{n_0} \delta_m \,t^{m} + \sum_{j=1}^M \left( \sum_{m=0}^{n_j}  \, \gamma_{j,m} \, t^{m} \right) {\mathrm e}^{\lambda_j t} , 
$$
with $\delta_{m} \in {\mathbb C}$, where the polynomial term occurs for $\lambda = 0$. Obviously, $\lambda_0=0 \in \frac{{\mathrm i}}{P} {\mathbb Z}$ for any $P>0$, and we need to apply the procedure described in Section \ref{sec:52} for its recovery.
Our models also cover real signals of the form 
\begin{equation}\label{real} y(t) = \sum_{j=1}^{M} \gamma_{j} \, \cos(2 \pi \alpha_{j}t + b_{j}) \end{equation}
with $\gamma_{j} \in {\mathbb R} \setminus \{0\}$, $\alpha_{j} \in {\mathbb R}$ and $b_{j} \in [0, 2\pi)$ considered in \cite{PP2020}.
The algorithms in \cite{PP2020} for recovery of $y(t)$ in (\ref{real}) are also based on rational approximation of the Fourier coefficients, but have essentially used  the additional information that all parameters are real, and are therefore different from the algorithms for the complex case considered here. Moreover, because of a different representation of the rational function in form of partial fraction decomposition the algorithm in \cite{PP2020} requires to use  modified  coefficients $\tilde{c}_k(f)=\mathrm{Re } \, c_k(f) +\frac{{\mathrm i }}{k} \, \mathrm{Im } \, c_k(f)$ instead of $c_k(f)$. Our approach can now 
 be also applied to the recovery of 
$$ y(t) = \sum_{j=1}^{M} \gamma_{j}(t) \, \cos(2 \pi \alpha_{j}t + b_{j}), $$
where $\gamma_{j}(t)$ are polynomial of finite degree. 

\section*{Acknowledgement}
The authors gratefully acknowledge support by the German Research Foundation in the framework of the RTG 2088.

\small
\bibliography{myBib}{}
\bibliographystyle{plain}

\end{document}